\documentclass{article}
\usepackage{makeidx}
\usepackage{amssymb , amsmath, amsthm,graphicx, float}
\usepackage{xcolor}
\usepackage{enumerate}
\newtheorem{defi}{Definition} 
\newtheorem{thm}[defi]{Theorem}
\newtheorem{theorem}[defi]{Theorem}
\newtheorem{example}[defi]{Example}

 \newtheorem{prop}[defi]{Proposition}

\newtheorem{lemma}[defi]{Lemma}
\newtheorem{cor}[defi]{Corollary}
 
\newcommand{\twosystem}[2]{\left\{\begin{aligned} &#1\\ &#2\end{aligned}\right.}

\newcommand{\nero}{\smallskip$\bullet\quad$\rm}

\newcommand{\matrice}{\begin{pmatrix}}
\newcommand{\ok}{\end{pmatrix}}

\newcommand{\eps}{\epsilon}

\newcommand{\scal}[2]{\langle{#1},{#2}\rangle}

\newcommand{\abs}[1]{\lvert{#1}\rvert}

\newcommand{\bd}{\partial}

\newcommand{\derive}[2]{\dfrac{\bd #1}{\bd#2}}

\DeclareMathOperator{\arctanh}{arctanh}

\DeclareMathOperator{\arccoth}{arcCoth}

\title{Geometry of the magnetic Steklov problem on Riemannian annuli}
\author{Luigi Provenzano\footnote{Sapienza Universit\`a di Roma\,, Dipartimento di Scienze di Base e Applicate per l'Ingegneria\,, Via Scarpa 16\,, 00161 Roma\,, Italy. Email:\, luigi.provenzano@uniroma1.it} and Alessandro Savo\footnote{Sapienza Universit\`a di Roma\,, Dipartimento di Scienze di Base e Applicate per l'Ingegneria\,, Via Scarpa 16\,, 00161 Roma\,, Italy. Email:\, alessandro.savo@uniroma1.it}}
\date{}

\begin{document}

\maketitle

\noindent
{\bf Abstract.} 
We study the geometry of  the first two eigenvalues of a magnetic Steklov problem on an annulus $\Sigma$ (a compact Riemannian surface with genus zero and two boundary components), the magnetic potential being the harmonic one-form having flux $\nu\in\mathbb R$ around any of the two boundary components. The resulting spectrum can be seen as a perturbation of the classical, non-magnetic Steklov spectrum, obtained when $\nu=0$ and studied e.g., by Fraser and Schoen in \cite{fs1,fs2}.
We obtain sharp upper bounds for the first and the second normalized eigenvalues and we discuss the geometry of the maximisers. 

 Concerning the first eigenvalue, we isolate a noteworthy class of maximisers which we call $\alpha$-surfaces: they are free-boundary surfaces which are stationary for a weighted area functional (depending on the flux) and have proportional principal curvatures at each point; in particular, they belong to the class of linear Weingarten surfaces. 
  
Inspired by \cite{fs1}, we then study the second normalized eigenvalue for a fixed flux $\nu$ and prove the existence of a maximiser for rotationally invariant metrics. Moreover, the corresponding eigenfunctions define a free-boundary immersion in the unit ball of $\mathbb R^3$. Finally, we prove that the second normalized eigenvalue associated to a flux $\nu$  has an absolute maximum when $\nu=0$, the corresponding maximiser being the critical catenoid. 


\vspace{11pt}


\vspace{6pt}

\tableofcontents

\section{Introduction}

This paper deals with a magnetic Steklov problem on a  Riemannian annulus $\Sigma$ with smooth boundary $\partial\Sigma$. For the precise definitions we refer to Subsection \ref{sub:intro:nomag} below. We are interested in sharp inequalities for the first two eigenvalues, and in the geometry of the optimizers. First, we review the main features of the classical (non-magnetic) Steklov problem on Riemannian surfaces. Then we give a short summary of the type of results that we obtain in the magnetic case.

\subsection{The non-magnetic case}\label{sub:intro:nomag} Let $\Sigma$ be a connected Riemannian surface with boundary $\partial\Sigma$. The classical Steklov problem (see \cite{steklov}) is:
\begin{equation}\label{steklov_c}
\begin{cases}
\Delta u=0 & {\rm in\ }\Sigma\\
\frac{\partial u}{\partial N}=\sigma u & {\rm on\ }\partial\Sigma.
\end{cases}
\end{equation}
Here $N$ is the exterior unit normal to $\bd\Sigma$ and $\Delta$ is the Laplace-Beltrami operator.
It is well-known that \eqref{steklov_c} admits a discrete sequence of non-negative eigenvalues of finite multiplicity, diverging to $+\infty$:
$$
0=\sigma_1(\Sigma)<\sigma_2(\Sigma)\leq\dots\leq\sigma_k(\Sigma)\leq\cdots
$$
The first eigenvalue $\sigma_1$ is zero, with associated eigenfunctions given by the constants.   In this paper the first {\it positive} eigenvalue of \eqref{steklov_c} will be numbered $\sigma_2(\Sigma)$: the reason is that in the magnetic case we will consider the first eigenvalue $\sigma_1$ which, very often, is {\it positive}. 

\medskip

An {\it isoperimetric inequality} for $\sigma_2$ was given by Weinstock (see \cite{weinstock}). To recall it, first observe that the $k$-th normalized  eigenvalue, distinguished by a bar:
$$
\bar\sigma_k(\Sigma):=|\partial\Sigma|\sigma_k(\Sigma)
$$
is invariant by homotheties.  

\begin{theorem}[Weinstock, 1954] If $\Sigma$ is a simply connected planar domain, then
$
\bar\sigma_2(\Sigma)\leq 2\pi.
$
Equality holds if and only if $\Sigma$ is a disk.
\end{theorem}

The upper bound is valid, more generally, for any simply connected Riemannian surface. For an arbitrary topology the following upper bound was proved in \cite{gipo}: 

\begin{theorem}[Girouard-Polterovich, 2012] If $\Sigma$ is a Riemannian surface with genus $\gamma$ and $k$ boundary components we have:
\begin{equation}\label{gp}
\bar\sigma_2(\Sigma)\leq 2\pi(\gamma+k).
\end{equation}
\end{theorem}

The previous bound can be refined in the case of annuli: by definition, a {\it Riemannian annulus} (or simply, an {\it annulus})  is a Riemannian surface having genus zero and two boundary components. 
It is a Riemannian surface diffeomorphic to the {\it flat annulus} $[0,1]\times\mathbb S^1$. In this case, \eqref{gp} reads $\bar\sigma_2\leq 4\pi$, which is not sharp because one has the following result proved in \cite{fs1}:

\begin{theorem}[Fraser-Schoen, 2011] Let $M_0$ be the positive root of the equation $x\tanh x=1$.  Then, for any rotationally invariant annulus one has:
$$
\bar\sigma_2(\Sigma)\leq \dfrac{4\pi}{M_0}=\bar\sigma_2(\Sigma_c)
$$
Equality holds for a particular rotation surface $\Sigma_c$, called {\rm critical catenoid}.
\end{theorem} 
Note that $M_0\approx 1.2$. After some deep additional work, the authors were able to extend the result to {\it all} annuli in \cite{fs2}:

\begin{theorem}[Fraser-Schoen, 2016] For any annulus $\Sigma$ one has
$\bar\sigma_2(\Sigma)\leq \bar\sigma_2(\Sigma_c)$. Hence the critical catenoid is an absolute maximum for the lowest (positive) normalized Steklov eigenvalue on annuli.
\end{theorem}

The critical catenoid has remarkable geometric properties: it is the unique rotational annulus which is minimal and {\it free  boundary} in the unit ball $B$ of $\mathbb R^3$, where free boundary means that $\partial\Sigma\subseteq\partial B$ and $\Sigma$ meets the boundary of $B$ orthogonally.
To construct $\Sigma_c$, consider the rotational surface in $\mathbb R^3$ obtained by rotating the curve (profile) $x=\cosh z$ around the $z$-axis. 
Draw the line from the origin which is tangent to the profile: the slice that we obtain is free boundary in a certain ball $B_R$ of radius $R$.  Scaling back $B_R$ to $B$ we obtain the critical catenoid. 
 
 \medskip
 
 Minimal free boundary immersions are stationary points of the area functional for variations of $\Sigma$ keeping its boundary on $\partial B$. 
The theory of minimal free boundary immersions is in some ways reminiscent of the theory of minimal spherical immersions, although it is very different in many respects. There are many well established results relating the maximisers for the $k$-th normalized eigenvalue of closed Riemannian surfaces to minimal immersions in $\mathbb S^p$ for some $p\in\mathbb N$  (see e.g., \cite{ei1,ei2,ly0,nad}). Fraser and Schoen \cite{fs2} showed that, when dealing with Riemannian surfaces with boundary, maximisers  for the $k$-th normalized eigenvalue of the Steklov problem give rise to  minimal free boundary immersions in some Euclidean unit ball.
\subsection{Perturbation of the Laplacian by a magnetic flux} 

In this paper we consider a Steklov spectrum by introducing the magnetic Laplacian associated to a closed potential one-form $A$ of flux $\nu\in\mathbb R$ on any annulus $\Sigma$, the flux being the circulation of $A$ around any of the two boundary curves of $\Sigma$ (the two circulations are the same because $A$ is closed). This leads to the perturbed spectrum 
$$
0\leq\sigma_1(\Sigma,\nu)\leq\sigma_2(\Sigma,\nu)\leq\cdots\leq\sigma_k(\Sigma,\nu)\leq\cdots
$$ 
which in fact depends only on the flux $\nu$ (by the gauge invariance property of the magnetic Laplacian) and which we will call {\it magnetic Steklov spectrum of flux $\nu$}. We denote it by ${\rm Spec}(\Sigma,\nu):=\{\sigma_k(\Sigma,\nu), k\in\mathbb N\}$. It has the property of being {\it periodic}: ${\rm Spec}(\Sigma,\nu)={\rm Spec}(\Sigma,\nu + n)$ for all integers $n$, and remains unchanged when $\nu$ is replaced by $-\nu$.
Hence, it is enough to study the spectrum  when the flux takes value in the interval 
$[0,\frac 12]$; in particular, if the flux is an integer, the spectrum reduces to the classical, non-magnetic Steklov spectrum $\{\sigma_k(\Sigma)\}_{k=1}^{\infty}$.

We will postpone the precise definition and the properties of the magnetic Laplacian and its Steklov spectrum to Section \ref{sec:mag:lap}.

\nero From now on, we assume $\nu\in [0,\frac 12]$. 

\medskip
The magnetic Steklov spectrum is a perturbation of the usual Steklov spectrum $\sigma_k(\Sigma):= \sigma_k(\Sigma,0)$ in the sense that $\sigma_k(\Sigma,\nu)\to \sigma_k(\Sigma)$ as $\nu\to 0$ (more generally, as $\nu$ tends to an integer). 

\medskip 

We now briefly describe the results that we obtain for $\sigma_1(\Sigma,\nu)$ and $\sigma_2(\Sigma,\nu)$. As in  the classical case, we are interested in the {\it normalized} eigenvalues $\bar\sigma_k(\Sigma,\nu):=|\partial\Sigma|\sigma_k(\Sigma,\nu)$, which turn out to be invariant by homotheties for all fixed $\nu$. 

\medskip

{\bf Upper bound of the first eigenvalue.} We first recall that in the non-magnetic case $\nu=0$ the first eigenvalue is always zero, with associated eigenspace being given by the constant functions; actually, we have that $\sigma_1(\Sigma,\nu)=0$ if and only if $\nu$ is an integer. Therefore, the geometric invariants arising in the study of $\sigma_1(\Sigma,\nu)$, which is positive when $\nu$ is not an integer, are peculiar to the magnetic situation, and do not have, at least a priori, a non-magnetic counterpart. Analogies with the classical case can be looked for when one examines the second eigenvalue $\sigma_2(\Sigma,\nu)$ which, when $\nu\in\mathbb Z$, reduces to the first positive non-magnetic Steklov eigenvalue.

Having said that we first give a sharp, conformal upper bound of $\bar\sigma_1(\Sigma,\nu)$ for all annuli $\Sigma$. The result reads:
$$
\bar\sigma_1(\Sigma,\nu)\leq 4\pi\nu \tanh(\nu M),
$$
where $M$ is the so-called {\it conformal modulus} of $\Sigma$ (see Subsection \ref{sub:annulus} for the definition). When $\nu$ is an arbitrary real number, the upper bound holds with $\min_{n\in\mathbb Z}\abs{n-\nu}$ replacing $\nu$ on the right hand side. This is presented in Theorem \ref{max_s1}.

\smallskip

Again, in the non-magnetic case $\nu=0$ both sides are zero, and there is nothing to say. 

\smallskip
 
The upper bound turns out to be exact asymptotically as $\nu\to 0$, in the sense that, for any annulus $\Sigma$, one has that $\bar\sigma_1(\Sigma,\nu)\sim 4\pi\nu \tanh(\nu M)$ as $\nu\to 0$: loosely speaking, this means that  one can ``hear''   the conformal class of $\Sigma$ knowing the lowest normalized eigenvalue for all $\nu$ sufficiently small. This result is presented in Theorem \ref{hear_conf}.

\smallskip

{\bf Upper bound of the second eigenvalue.} Then, in Theorem \ref{max_RI} we give a sharp upper bound of the second normalized eigenvalue for any value of the flux and for all {\it rotationally invariant annuli}. It is enough to state it when $\nu\in[0,\frac{1}{2}]$:
\begin{equation}\label{ubsigmatwo}
\bar\sigma_2(\Sigma,\nu)\leq\begin{cases}
 \dfrac{4\pi}{M_0}& {\rm when\ } \nu=0\\
 4\pi\nu \tanh(\nu M^*) & {\rm when \ }\nu\in (0,\frac 12)\\
2\pi & {\rm when\ } \nu=\frac 12
\end{cases}
\end{equation}
When $0<\nu<\frac 12$, the constant $M^*=M^*(\nu)>0$ denotes the unique positive root  of  the equation $(1-\nu)\tanh((1-\nu)x)=\nu\coth(\nu x)$  (it tends to $+\infty$ when $\nu\to\frac 12$).   In particular we show that, if the flux is not a half-integer, a maximiser for $\bar\sigma_2(\Sigma,\nu)$ exists in the class of all rotationally invariant annuli. Of course, the result when $\nu=0$ is due to Fraser and Schoen \cite{fs1}.
The upper bound in  \eqref{ubsigmatwo} is continuous and decreasing in $\nu$ for $\nu\in [0,\frac 12]$ (see Corollary \ref{cor_max_RI}), hence we get:
$$
\bar\sigma_2(\Sigma,\nu)\leq \bar\sigma_2(\Sigma_c,0)=\frac{4\pi}{M_0}\approx\dfrac{4\pi}{1.2},
$$
which implies the following fact.

\begin{theorem}
 The critical catenoid maximizes $\bar\sigma_2(\Sigma,\nu)$ among all rotationally invariant annuli, for all values of the flux $\nu\in\mathbb R$.
\end{theorem}

This generalises \cite{fs1} (i.e., $\nu=0$) to the magnetic case. We conjecture that the theorem holds for all annuli, and not just in the rotationally invariant case. 
 
\medskip

{\bf Geometry of maximisers.} We then study the geometry of maximisers and ask the following question:

\nero are the maximisers (conformally) minimal in some sense and free boundary in the unit ball of $\mathbb R^3$? 

\medskip

For the first eigenvalue $\bar\sigma_1(\Sigma,\nu)$, we show that the maximisers are conformally equivalent to free boundary rotation surfaces which are minimal for the weighted area
$$
{\rm Area}_{\alpha}(\Sigma):=\int_{\Sigma}\rho^{\alpha-1}\,dv
$$
for some value $\alpha>0$ depending on the conformal modulus of $\Sigma$, where $\rho$ is the distance to the axis of rotation (see Theorem \ref{asurf-s1}). 
We then have a family of distinguished free boundary surfaces, which we call critical $\alpha$-surfaces: they foliate the unit ball in $\mathbb R^3$ and have interesting geometric properties, for example their principal curvatures $\kappa_1,\kappa_2$ are proportional at every point ($\kappa_1+\alpha \kappa_2=0$ where $\kappa_1$ is the meridian curvature). In particular, these maximisers are a particular case of surfaces already known to differential geometers, called Weingarten surfaces. The critical catenoid is the member of the family corresponding to $\alpha=1$. We give a list of equivalent characterizations of $\alpha$-surfaces (one of them involving the magnetic Laplacian) in Theorem \ref{geom-asurf}.

\smallskip 

Concerning the maximisers of $\bar\sigma_2(\Sigma,\nu)$, we observe that the multiplicity of the corresponding eigenvalue is $2$, and the (suitably normalized) eigenfunctions define a free boundary immersion in the unit ball which is minimal in an appropriate sense, see Theorems \ref{FSs2} and \ref{FSmagnetic}. We plan to investigate the variational meaning of this notion of minimality in a forthcoming paper.

\medskip In this paper we have considered the case of Riemannian annuli, that is, Riemannian surfaces of genus $0$ and $2$ boundary components. The easiest setting for studying maximisers for magnetic Steklov eigenvalues and their geometry is the simply connected case. In this case, the analogous inequality of Weinstock for $\bar\sigma_1(\Sigma,\nu)$  has been proved in \cite{cps}: namely, among all simply connected Riemannian surfaces, the first normalized eigenvalue is maximised (up to $\sigma$-homotheties) by a planar disk centered at the pole of the magnetic potential $A=\frac{\nu}{x^2+y^2}(-ydx+xdy)$, and the maximum value is $2\pi\nu$.

\medskip The present paper is organized as follows. In Section \ref{sec:mag:lap} we introduce the magnetic Laplacian on Riemannian surfaces, in particular on Riemannian annuli, and we list some relevant properties and definitions. In Section \ref{spec:rot} we compute explicitly the spectrum of rotationally invariant, symmetric annuli and identify the first and second eigenvalue. In Section \ref{sec:bound:first} we state the results on the maximisation of the first normalized eigenvalue, while in Section \ref{sec:bound:second} we state the results on the maximisation of the second normalized eigenvalue. In Section \ref{geom1} we discuss the geometry of maximisers of $\bar\sigma_1$ while in Section \ref{geom2} we discuss the geometry of maximisers of $\bar\sigma_2$. All the proofs are contained in a series of appendices. Namely, the proofs of the results of Sections \ref{sec:mag:lap}, \ref{sec:bound:first}, \ref{sec:bound:second}, \ref{geom1}, \ref{geom2} are contained, respectively, in Appendices \ref{app:2}, \ref{app:4}, \ref{app:5}, \ref{app:6}, \ref{app:7}.

\section{The magnetic Laplacian on Riemannian annuli}\label{sec:mag:lap}

\subsection{The magnetic Laplacian}
Let $\Sigma$ be a Riemannian surface, with or without boundary, and let $A$ be a smooth {\it real} one-form on $\Sigma$. The two-form $\beta:= dA$ is called the {\it magnetic field} and $A$ is the {\it potential one-form}. 
The magnetic Laplacian associated with the pair $(\Sigma,A)$ is denoted by $\Delta_A$ and acts on smooth complex valued functions $u$ as follows:
$$
\Delta_Au=\delta^Ad^Au
$$
where $d^Au=du-iuA$ is the {\it magnetic differential} and $\delta^A\omega=\delta\omega+i\omega(A^{\sharp})$ for all one-forms $\omega$,  where $A^{\sharp}$ is the dual vector field of $A$. $\delta^A\omega$ is also called the {\it magnetic co-differential}. More explicitly:
$$
\Delta_Au=\Delta u+\abs{A}^2u+2i\scal{A}{du}-i u\delta A.
$$

\medskip

The magnetic Steklov problem amounts to finding all $\Delta_A$-harmonic functions $u$ and all eigenvalues $\sigma$ satisfying the problem:
\begin{equation}\label{steklov_mag}
\begin{cases}
\Delta_Au=0 & {\rm on\ }\Sigma\\
d^Au(N)=\sigma u & {\rm on\ }\partial\Sigma,
\end{cases}
\end{equation}
where $N$ is the outer unit normal to $\Sigma$. It is well-known that the spectrum is non-negative, discrete, and increasing to $+\infty$; we stress that it depends on the pair $(\Sigma,A)$ and, of course, not only on $\Sigma$:
$$
0\leq\sigma_1(\Sigma, A)\leq\sigma_2(\Sigma,A)\leq\dots\leq\sigma_k(\Sigma,A)\leq\cdots
$$
From a variational point of view, the magnetic Laplacian is naturally associated to the energy quadratic form
$$
\mathcal Q(u)=\int_{\Sigma}\abs{d^Au}^2\,dv,
$$
and hence we have the usual min-max characterization of the eigenvalues:
\begin{equation}\label{minmax}
\sigma_k(\Sigma,A)=\min_{\substack{U\subset H^1(\Sigma)\\{\rm dim}U=k}}\max_{0\ne u\in U}\frac{\mathcal Q(u)}{\int_{\partial\Sigma}|u|^2ds},
\end{equation}
where $H^1(\Sigma)$ is the usual Sobolev space of functions in $L^2(\Sigma)$ with first order derivatives in $L^2(\Sigma)$. Of course, the Steklov problem when $A=0$ reduces to the classical (non-magnetic) Steklov problem. Actually, this happens in greater generality. 
In fact, {\it gauge invariance} implies that the spectrum of $(\Sigma, A)$ is the same as the spectrum of $(\Sigma, A+df)$ for any  real function $f$ on $\Sigma$, which follows from the fact that $\Delta_A$ and $\Delta_{A+df}$ are unitarily equivalent.  Thus, in treating the eigenvalues, we have the freedom to alter the potential one-form by any exact one-form. 

A bit more generally, we have gauge invariance also when $A$ is closed and the flux of $A$ around any closed curve in $\Sigma$ is an integer: by definition, if $c:[0,1]\to\Sigma$ is a loop (a continuous curve such that  $c(0)=c(1)$) the number
$$
\Phi (A,c)=\dfrac{1}{2\pi}\int_0^1A(c'(t))dt
$$
is called the {\it flux of $A$ along $c$}. We remark that the orientation of the loop does not affect the spectrum, hence we do not need to specify it. Summarizing, we have the following fact: 

\begin{prop} In the above notation, we have {\it ${\rm spec}(\Sigma,A)={\rm spec} (\Sigma,0)$ if and only if $A$ is closed and the flux of $A$ around any closed curve is an integer.} 
\end{prop}

This fact has been observed in \cite{shig} and later proved in \cite{hho2}.

\nero In this paper we shall investigate the case when the magnetic field is {\it zero}, i.e., when the potential one-form $A$ is {\it closed}.

\subsection{The magnetic Laplacian with zero magnetic field on an annulus}\label{sub:annulus}

The purpose of this paper is to investigate the magnetic Steklov problem when the one-form $A$ is {\it closed} and $\Sigma$ is an annulus, that is, a Riemannian surface diffeomorphic to
$$
T=[0,1]\times\mathbb S^1.
$$
Then, $\Sigma$ has genus zero and two boundary components. 
The absolute de Rham real cohomology in degree one of an annulus is one-dimensional. Cohomology classes are indexed by the flux of closed one-forms around any of the two boundary components
$\bd_1T=\{0\}\times\mathbb S^1, \bd_2T=\{1\}\times\mathbb S^1$: in fact, these two fluxes are the same because $A$ is closed, by the Stokes theorem.
Let $A, A'$ be  any two closed one-forms with fluxes $\nu,\nu'$ which differ by an integer; if one considers the magnetic Laplacians associated to the pairs $(\Sigma,A)$ and $(\Sigma,A')$,  we see they are unitarily equivalent by the Shigekawa argument discussed in the previous subsection (see \cite{shig}, see also \cite{cps}), and therefore have the same Steklov spectrum.
Then, writing 
$$
\sigma_k(\Sigma,\nu)
$$
we will denote the $k$-th eigenvalue of the Steklov problem associated with {\it any} closed potential $A$ of flux $\nu$, without ambiguity. In particular, note that, if $\nu$ is not an integer, then
$$
\sigma_1(\Sigma,\nu)>0,
$$
while in the classical Steklov case ($\nu=0$) we have of course $\sigma_1(\Sigma,0)=0$ with associated eigenspace  given by the constant functions. We will keep this notation, so that, according to this numbering, the writing $\sigma_2(\Sigma,0)$ will denote the {\it first positive} eigenvalue of the usual non-magnetic Steklov problem.  

\medskip

Note that since $A$ is closed the magnetic field is zero; however, as we have said, if $\nu\notin\mathbb Z$, the ground state energy $\sigma_1(\Sigma,\nu)$ is positive: this phenomenon is vaguely referred, in the physics literature, as the {\it Aharonov-Bohm effect}  (for a discussion of this physical phenomenon, we refer to \cite{AhBo}). We can always assume, by gauge invariance, that $\nu\in [-\frac 12,\frac 12]$; but in fact, since replacing $\nu$ by $-\nu$ the spectrum remains invariant (and the eigenfunctions change to their conjugates), we can assume
$$
\nu\in \left[0,\frac 12	\right].
$$

\begin{defi}[Conformal modulus] Any annulus $\Sigma$  is conformally equivalent to the cylinder 
$$
[-M,M]\times \mathbb S^1,
$$
for a unique $M=M(\Sigma)$,  which is called the {\rm conformal modulus} of $\Sigma$. In other words, any annulus is isometric to $[-M,M]\times \mathbb S^1$ endowed with the metric 
$
g=\psi(t,\theta)(dt^2+d\theta^2),
$
for a smooth $\psi:[-M,M]\times\mathbb S^1$.
\end{defi} 

\medskip

If the metric of $\Sigma$ is rotationally invariant (i.e., it does not depend on $\theta$), then so does the conformal factor $\psi$. In particular, to any surface of revolution in $\mathbb R^3$ can be associated a metric of type: $\psi(t)(dt^2+d\theta^2)$
on some $[-M,M]\times\mathbb S^1$.

\subsection{Conformal invariance of magnetic Dirichlet energy} \label{conf:inv}

We recall here a fundamental property of the magnetic Dirichlet energy, analogous to the conformal invariance of the usual Dirichlet energy in two dimensions. Let $(\Gamma, A)$ be a magnetic pair (surface and real one-form) and let
$
\phi:\Sigma\to\Gamma
$
be a conformal map of Riemannian surfaces. Then for all $u\in C^{\infty}(\Gamma,\mathbb C)$ one has:
$$
\int_{\Sigma}|d^{\phi^*A}(\phi^*u)|^2\,dv=\int_{\Gamma}|d^{A}u|^2dv
$$
where $\phi^*A$ is as usual the potential one-form on $\Sigma$ given by pull-back of $A$ and $\phi^*u:=u\circ\phi$. For a proof of this result we refer to \cite{cps}.

\subsection{Normalized eigenvalues and $\sigma$-homotheties} Let $\{\sigma_k(\Sigma,\nu)\}_{k=1}^{\infty}$ be the magnetic Steklov spectrum of a Riemannian annulus $\Sigma$ with flux $\nu$. The $k$-th normalized eigenvalue is by definition $\bar\sigma_k(\Sigma,\nu):=|\partial\Sigma|\sigma_k(\Sigma,\nu)$.
It is easy to see that it is invariant by homotheties, exactly as in the non-magnetic case, but even more is true.
\begin{defi}[$\sigma$-homothety]
Given two surfaces with boundary $\Sigma, \Gamma$,
 a {\it $\sigma$-homothety} between them is a conformal diffeomorphism $\phi:\Sigma\to\Gamma$ which restricts to a homothety on $\bd\Sigma$. Equivalently, the map
$$
\phi:(\Sigma, g)\to (\Gamma,h)
$$
($g$ and $h$ being the respective metrics) is a $\sigma$-homothety if and only if
$
 \phi^*h=\psi g,
$
for a smooth function $\psi$ on $\Sigma$, and the conformal factor $\psi$ is constant on the boundary of $\Sigma$. 
\end{defi}
The following result is well-known in the non-magnetic case, but extends also to any flux $\nu$.
 
 \begin{prop}\label{sigmah} If two annuli $\Sigma,\Gamma$ are $\sigma$-homothetic, then they have the same normalized spectrum for all fluxes $\nu$:
$$
\bar\sigma_k(\Sigma,\nu)=\bar\sigma_k(\Gamma,\nu)
$$
for all $k\in\mathbb N$.
 \end{prop}
The proof is presented in Appendix \ref{app:2}.

 \subsection{Rotationally invariant annuli} Let $\Sigma$ be a rotationally invariant annulus with conformal modulus $M$: then, it is isometric with $[-M,M]\times\mathbb S^1$ endowed with the metric 
$
\psi(t)(dt^2+d\theta^2),
$
where now $\psi:[-M,M]\to (0,\infty)$. 
In addition, we say that $\Sigma$ is {\it symmetric} if $\psi(-M)=\psi(M)$ (this is true, in particular, if $\psi$ is an even function: $\psi(t)=\psi(-t)$).  This implies that the two boundary components have the same length, 
which guarantees that $\Sigma$ is $\sigma$-homothetic with the cylinder $[-M,M]\times\mathbb S^1$. We remark that any surface in $\mathbb R^3$ obtained by rotating around the $z$-axis  the graph of $x=\rho(z)$ on  $[-M,M]$, with $\rho$ even, is a symmetric annulus. 

\medskip

From Proposition \ref{sigmah} we deduce the following:

\begin{prop}\label{RIS-sigmah} Any rotationally invariant, symmetric annulus of conformal modulus $M$ has the same normalized spectrum of the cylinder $[-M,M]\times \mathbb S^1$.
\end{prop} 

\section{Normalized spectrum of rotationally invariant, symmetric annuli}\label{spec:rot}

In this section we compute explicitly the magnetic Steklov spectrum with zero magnetic field of rotationally invariant, symmetric annuli for closed potential one-forms. 

\subsection{Flat annuli}

A {\it  flat annulus} is the product $\Gamma_M=[-M,M]\times\mathbb S^1$ with coordinates $(t,\theta)\in[-M,M]\times\mathbb S^1$, endowed with the flat metric $dt^2+d\theta^2$. We take as magnetic potential the harmonic one-form $A=\nu d\theta$ which has flux $\nu$ around $\mathbb S^1$. We can think of $\Gamma_M$ as a cylinder in $\mathbb R^3$, and in this case the one-form $A$ can be expressed in Cartesian coordinates as $A=\frac{\nu}{x^2+y^2}(-ydx+xdy)$ (and is defined on the whole $\mathbb R^3$). Since $A$ is harmonic (hence co-closed) the magnetic Laplacian is given by:
$$
\Delta_Au=\Delta u+\abs{A}^2 u+2i\langle du, A\rangle.
$$
We now look for solutions of $\Delta_Au=0$; as customary we separate variables, and consider functions of the form $u(t,\theta)=\phi(t)e^{ik\theta}$, $k\in\mathbb Z$. Since the metric is flat we have
$$
\Delta u=\left(-\phi''(t)+k^2\phi(t)\right)e^{ik\theta}.
$$
Moreover, $|A|^2=\nu^2$ and $du=e^{ik\theta}(\phi'(t)dt+ik\phi(t)d\theta)$, so that
$$
\langle du, A\rangle=i\nu k \phi(t) e^{ik\theta}.
$$
It follows that
$$
\Delta_Au=\left(-\phi''+(k-\nu)^2\phi(t)\right)e^{ik\theta}.
$$
So, if we want $\Delta_Au=0$, we must solve the ODE  $\phi''=(k-\nu)^2\phi$ on the interval $[-M,M]$. The general solution is
$$
\phi(t)=a\cosh(\abs{k-\nu}t)+b\sinh(\abs{k-\nu}t)\,,\ \ \ a,b\in\mathbb R,
$$
when $k\ne\nu$. If $\nu$ is an integer and $k=\nu$ we have $\phi(t)=a+bt$, $a,b\in\mathbb R$. Since the case $\nu\in\mathbb Z$ is exactly the classical Steklov problem, for which the spectrum on cylinders has been extensively described (see e.g., \cite{gipo_sur}), we shall assume from now on (unless explicitly stated) that $\nu\notin\mathbb Z$. In the following subsection we consider separately the linearly independent solutions $\cosh(|k-\nu|t)$ (solutions of first type) and $\sinh(|k-\nu|t)$ (solutions of second type).

\medskip

{\bf Solutions of the first type.}
The first type of solution is $u(t,\theta)=\cosh(\abs{k-\nu}t)e^{ik\theta}$. Note that $N=\partial_t$ on $t=M$ and $N=-\partial_t$ on $t=-M$, so that we easily compute
$$
\frac{\partial u}{\partial N}=|k-\nu|\tanh(|k-\nu|M)u.
$$
Hence $u$ is a Steklov eigenfunction associated to
$$
\sigma=\abs{k-\nu}\tanh(\abs{k-\nu}M).
$$

\medskip

{\bf Solutions of the second type.} We now take $u(t,\theta)=\sinh(\abs{k-\nu}t)e^{ik\theta}$.
Then:
$$
\frac{\partial u}{\partial N}=|k-\nu|\coth(|k-\nu|M)u
$$
and $u$ is a Steklov eigenfunction associated to 
$$
\sigma=|k-\nu|\coth(|k-\nu|M).
$$
A standard argument allows to prove that these solutions are all the magnetic Steklov eigenfunctions on $\Gamma_M$ (see \cite{cps} for a proof of this fact in the case of a disk, see also \cite{gipo_sur}). We have the following:

\begin{thm} The Steklov spectrum of the magnetic Laplacian on the flat cylinder $\Gamma_M=[-M,M]\times\mathbb S^1$ when  $\nu\notin \mathbb Z$ is given by the collection of all $\sigma_{1k}, \sigma_{2k}$ where
$$
\sigma_{1k}=\abs{k-\nu}\tanh(\abs{k-\nu}M), \quad \sigma_{2k}=\abs{k-\nu}\coth(\abs{k-\nu}M), \quad k\in\mathbb Z.
$$
\end{thm}

When $\nu=0$ (or, more in general, when $\nu\in\mathbb Z$) we know that the spectrum is given by
$$
0, \frac 1M, k\tanh(kM), k\coth(kM), \quad k=0,1,2,\dots
$$

\subsection{Normalized spectrum of rotationally invariant, symmetric annuli}

Throughout this section, we denote by $\Sigma(M)$ any rotationally invariant symmetric annulus having conformal modulus $M$ (an example is the flat annulus $\Gamma_M$ discussed in the previous subsection). 

\medskip

Since the boundary of $\Gamma_M$ consists of two circles of radius one, and since any rotationally invariant symmetric annulus of conformal modulus $M$ has the normalized spectrum of $\Gamma_M$ (see Proposition \ref{RIS-sigmah}), we have the following description of the normalized spectrum: note that it depends only on $M$ and $\nu$. 
\begin{thm} Let $\Sigma(M)$ be a rotationally invariant, symmetric annulus of conformal modulus $M$. Then its normalized spectrum $\{\bar\sigma_k(\Sigma(M),\nu)\}_{k\in\mathbb N}$ is given by the collection of all $\sigma_{1k}, \sigma_{2k}$ where:
$$
\bar\sigma_{1k}=4\pi\abs{k-\nu}\tanh(\abs{k-\nu}M), \quad \bar\sigma_{2k}=4\pi\abs{k-\nu}\coth(\abs{k-\nu}M), \quad k\in\mathbb Z.
$$
If $\nu=0$, we have $\bar\sigma_{10}=0$, $\bar\sigma_{20}=\frac{4\pi}{M}$.
\end{thm}

 We now discuss the magnetic case: hence we can restrict $\nu$ to the interval $(0,\frac 12]$.  We identify the first and the second eigenvalue as follows. 

\begin{prop}[First eigenvalue] Assume $\nu\in (0,\frac 12]$. The first normalized eigenvalue of a rotationally invariant, symmetric annulus $\Sigma(M)$ of conformal modulus $M$ is
$$
\bar\sigma_1(\Sigma(M),\nu)=\bar\sigma_{10}=4\pi\nu\tanh(\nu M).
$$
\end{prop}

Again note that when $\nu=0$, $\bar\sigma_1(\Sigma(M),0)=0$, with corresponding constant eigenfunctions.
When $\Sigma=[-M,M]\times \mathbb S^1$ the associated eigenfunction is $u(t,\theta)=\cosh(\nu t)$, which is real and radial, in the sense that it does not depend on the angular variable $\theta$; then (see the proof of Proposition 10) the same happens for any rotationally invariant, symmetric annulus $\Sigma(M)$.

\medskip

Now we turn to the second eigenvalue.  There are two competing branches:
\begin{equation}\label{branches}
\bar\sigma_{11}(\Sigma(M),\nu)=4\pi(1-\nu)\tanh((1-\nu)M), \quad \bar\sigma_{20}(\Sigma(M),\nu)=4\pi\nu\coth(\nu  M)
\end{equation}
 so that we have:
 
\begin{prop}[Second eigenvalue] \item a) Fix $\nu\in (0,\frac 12)$. The second normalized eigenvalue of a rotationally invariant, symmetric annulus $\Sigma(M)$ of conformal modulus $M$ is
$$
\bar\sigma_2(\Sigma(M),\nu)=\min\{\bar\sigma_{11},\bar\sigma_{20}\}=\min\{4\pi(1-\nu)\tanh((1-\nu) M),4\pi\nu\coth(\nu M)\}.
$$
\item b) If $\nu=\frac 12$ then:
$$
\bar\sigma_2(\Sigma(M),\frac 12)=2\pi\tanh(\frac M2).
$$
In that case $\bar\sigma_1=\bar\sigma_2$ has multiplicity $2$ for all values of the conformal modulus $M$. 
\end{prop}

We now fix $\nu$ and discuss the supremum of $\bar\sigma_2$ as a function of the conformal modulus $M$, that is, we set:
$$
\sigma_2^*(\nu):=\sup_{M\in (0,+\infty)}\bar\sigma_2(\Sigma(M),\nu)
$$
and, for $\nu\in (0,\frac 12)$, we let $M^*(\nu)$ be the unique positive root of the equation $(1-\nu)\tanh((1-\nu)x)=\nu\coth(\nu x)$. 

\begin{prop}\label{max_RIS} Let $\Sigma(M)$ be a rotationally invariant, symmetric annulus of conformal modulus $M$. 

\item (a) Fix $\nu\in (0,\frac 12)$. Then 
$
\sigma_2^*(\nu)
$
is attained for the unique value $M=M^*(\nu)$; that is 
$$
\sigma_2^*(\nu)=4\pi(1-\nu)\tanh((1-\nu)M^*(\nu))=4\pi\nu\coth( \nu M^*(\nu)).
$$
For this value $M=M^*(\nu)$ the second eigenvalue has multiplicity $2$, and we have $\bar\sigma_1<\bar\sigma_2=\bar\sigma_3$. 

\item (b) Fix $\nu=\frac 12$. Then $\sigma^*_2(\frac 12)$ is not attained for any finite value of $M$. In fact $\sigma^*_2(\frac 12)=2\pi$, and is reached asymptotically as $M\to+\infty$. Moreover, for any finite value of $M$ one has $\bar\sigma_1=\bar\sigma_2<\bar\sigma_3$.
\end{prop}

For the proof of (a), just observe that, as $M$ increases, the first branch in \eqref{branches} increases from $0$ to $4\pi(1-\nu)$, while the second decreases from $+\infty$ to $4\pi\nu$: since $4\pi\nu<4\pi(1-\nu)$, the two branches meet precisely when $M=M^*(\nu)$, which realizes the maximum.

 When $\nu=\frac 12$ we see immediately that $\bar\sigma_{11}<\bar\sigma_{02}$ for all $M$; the sup is attained only when $M\to\infty$ and takes the value $2\pi$.

\smallskip

Finally, one can easily check that when $\nu=0$ the two competing branches are $\bar\sigma_{11}(\Sigma(M),0)=4\pi\tanh(M)$ and $\bar\sigma_{20}(\Sigma(M),0)=\frac{4\pi}{M}$, and in this case $M^*(0)=M_0$ is the unique positive root of $x\tanh(x)=1$.

\begin{figure}[htp]

\centering
\includegraphics[width=.33\textwidth]{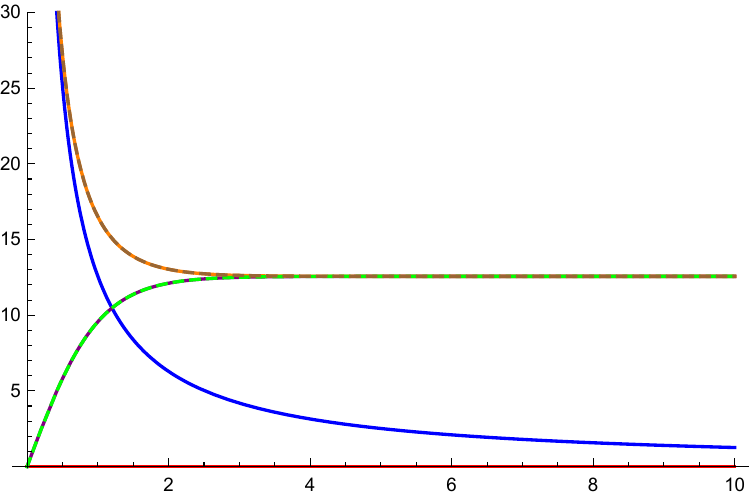}\hfill
\includegraphics[width=.33\textwidth]{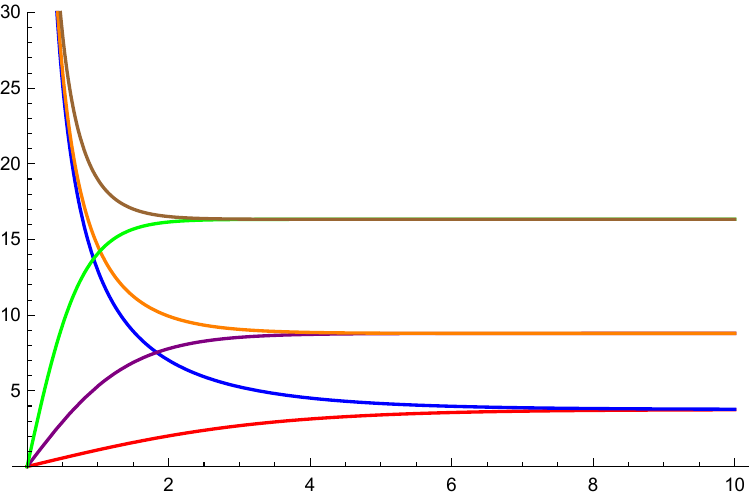}\hfill
\includegraphics[width=.33\textwidth]{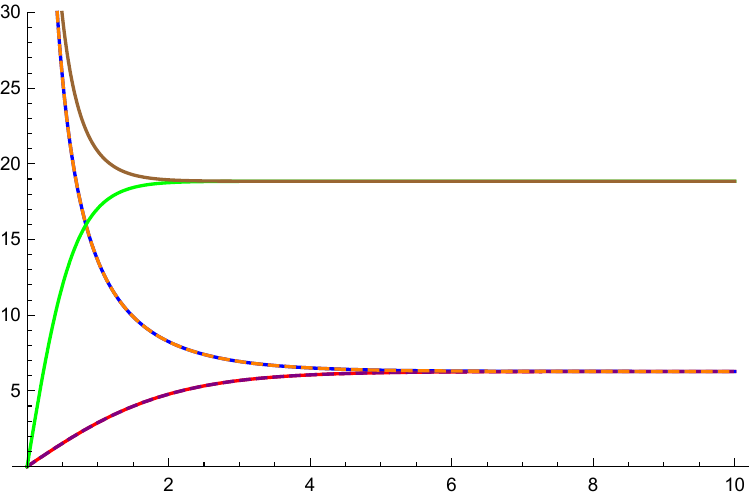}

\caption{From the left: $\nu=0,\nu=\frac{1}{3}$ and $\nu=\frac{1}{2}$. Colors: $\bar\sigma_{10}$ red; $\bar\sigma_{20}$ blue; $\bar\sigma_{11}$ purple; $\bar\sigma_{21}$ orange; $\bar\sigma_{1,{-1}}$ green;  $\bar\sigma_{2,{-1}}$ brown.}
\label{fig:figure3}

\end{figure}

\section{Sharp upper bound for the first normalized eigenvalue}\label{sec:bound:first}

In this section we consider inequalities for the first normalized eigenvalue. The first result is a sharp upper bound valid for {\it any} annulus (not necessarily rotationally invariant).

\begin{theorem}\label{max_s1} Let $\Sigma$ be a Riemannian annulus. Then, for all $\nu\in [0,\frac 12]$:
\begin{equation}\label{fne}
\bar\sigma_1(\Sigma,\nu)\leq 4\pi\nu\tanh( \nu M(\Sigma))
\end{equation}
where $M(\Sigma)$ is the conformal modulus of $\Sigma$.  If $\nu\ne 0$, equality holds if and only if $\Sigma$ is $\sigma$-homothetic to the cylinder $\Gamma_M=[-M(\Sigma),M(\Sigma)]\times \mathbb S^1$. In particular, the  following uniform upper bound holds:
$$
\bar\sigma_1(\Sigma,\nu)<4\pi\nu
$$
independently of the conformal modulus. 
\end{theorem}

We have seen that the estimate $\bar\sigma_1(\Sigma,\nu)\leq 4\pi\nu\tanh( \nu M(\Sigma))$ becomes an equality for all rotationally invariant, symmetric annuli.
If we don't fix the conformal modulus, then the upper bound $4\pi\nu$ is never attained, and it is reached asymptotically by families of rotationally invariant, symmetric annuli of conformal modulus $M\to+\infty$.

We remark once again that if  $\nu$ is arbitrary the upper bound holds with $\min_{n\in\mathbb Z}\abs{n-\nu}$ replacing $\nu$ in the right hand side. 

The proof is presented in Appendix \ref{app:4}. 

In the rest of this subsection we discuss some remarks on the sharpness of the estimate of Theorem \ref{max_s1}.

\subsection{Asymptotics of the first normalized eigenvalue as $\nu\to 0$}

A first remark is that in fact every Riemannian annulus realizes equality in \eqref{fne} {\it asymptotically }as $\nu\to 0$, in the following sense:

\begin{theorem}\label{hear_conf} For any Riemannian annulus $\Sigma$ (not necessarily rotationally invariant) one has:
$$
\lim_{\nu\to 0}\dfrac{1}{\nu^2}\bar\sigma_1(\Sigma,\nu)=4\pi M(\Sigma).
$$
That is, asymptotically as $\nu\to 0$:
$$
\begin{aligned}
\bar\sigma_1(\Sigma,\nu)&= 4\pi M(\Sigma)\nu^2+o(\nu^2)\\
&=4\pi\nu\tanh(\nu M(\Sigma))+o(\nu^2)
\end{aligned}
$$
\end{theorem}

For the proof we refer to  Appendix \ref{app:4}. This inequality says that the magnetic Steklov spectrum detects the conformal modulus of an annulus, in the sense that if we know $\bar\sigma_1(\Sigma,\nu_n)$ for a sequence of fluxes $\nu_n\to 0$, then we also know  the conformal modulus of $\Sigma$.

\subsection{The case of planar annuli}

Another remark concerns the inequality $\bar\sigma_1(\Sigma,\nu)<4\pi\nu$ of Theorem  \ref {max_s1} for all $\nu\in [0,\frac 12]$: in this subsection, we examine the possibility of improving it when $\Sigma$ is a planar annulus. 

 Let $A(r,R)$ be the concentric annulus in the plane bounded by the concentric circles of radii $r$ and $R$: $A(r,R)=\{x\in\mathbb R^2:r<|x|<R\}$ centered at the origin. In this case we can take, as magnetic potential, the one-form $A=\nu d\theta$, where $\theta$ is the angular variable in polar coordinates in $\mathbb R^2$: in fact, $A$ is closed and has flux $\nu$ around both circles bounding $A(r,R)$.  Simple explicit calculations (which we omit for brevity) show that
$$
\bar\sigma_1(A(r,R),\nu)<2\pi\nu=\bar\sigma_1(D,\nu)
$$
where $D$ is any disk centered at $0$. Moreover $\bar\sigma_1(A(r,R),\nu)\to 2\pi\nu$ as $r\to 0^+$. This simple observation might suggest the following conjecture: 

\nero {\it is it true that every doubly connected planar domain $\Omega$ has first normalized Steklov eigenvalue smaller than $2\pi\nu$?} 

\smallskip

A bit surprisingly, one has the following result:

\begin{theorem}\label{planar} There exists a sequence of doubly connected planar domains $\{\Omega_{\epsilon}\}_{\epsilon\in(0,1)}$ such that, for any fixed $\nu\in [0,\frac 12]$:
$$
\lim_{\epsilon\to 0}\bar\sigma_1(\Omega_{\epsilon},\nu)=4\pi\nu.
$$
\end{theorem}
For the proof we refer to Appendix \ref{app:4}.  In other words, there exist sequences of planar annuli which saturate the inequality 
$\bar\sigma_1(\Sigma,\nu)<4\pi\nu$. As a consequence of Theorem \ref{max_s1}, it follows that the conformal modulus $M(\Omega_{\epsilon})$ of the annuli in these families tends to $+\infty$ as $\epsilon\to 0$. From the proof it can be easily deduced that the ratio of the two boundary lengths of $\Omega_{\epsilon}$ tends to $1$ as $\epsilon\to 0$. Again, this is not a surprise in view of Theorem \ref{max_s1}. A prototypical example of a domain $\Omega_{\epsilon}$ is illustrated in Figure \ref{fig_planar}.

\begin{figure}[htp]
\centering
\includegraphics[width=0.3\textwidth]{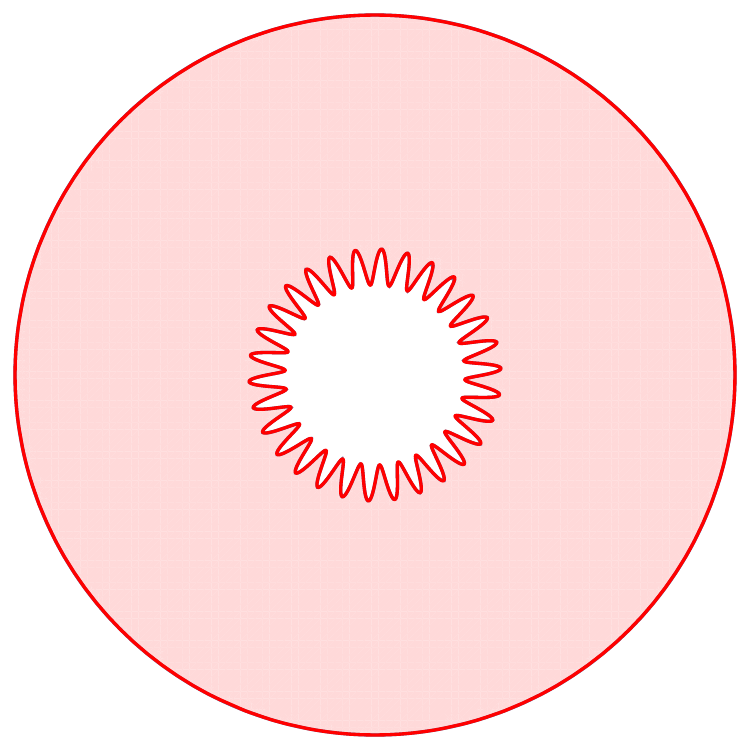}
\caption{A planar annulus $\Omega_{\epsilon}$ with first normalized eigenvalue close to $4\pi\nu$. The outer boundary is a unit circle, while the inner boundary is described by an oscillating function over a circle of radius $r_0(\epsilon)<1$. As $\epsilon\to 0$ we have that simultaneously $r_0(\epsilon)\to 0$, the width of the oscillations tends to zero, while their frequency tends to $+\infty$, in a suitable way (see the proof of Theorem \ref{planar} in Appendix \ref{app:4}).}
\label{fig_planar}

\end{figure}

\section{Upper bound for the second normalized eigenvalue for rotationally invariant annuli
}\label{sec:bound:second}

Through all this section we assume $\nu\in (0,\frac 12)$. As we have already mentioned, the case $\nu=0$ is the classical case of the Laplacian and has been treated in \cite{fs1}. For $\nu=\frac{1}{2}$ we have seen in Section \ref{spec:rot} that the first eigenvalue of a rotationally invariant, symmetric annulus is always double, and it has no maximum if we do not restrict to a fixed conformal class of annuli. 

\medskip

Recall that we have defined
$$
\sigma_2^*(\nu)=\max_M\bar\sigma_2(\Sigma(M),\nu).
$$  
where $\Sigma(M)$ denotes a rotationally invariant symmetric annulus of conformal modulus $M$: every such domain is in fact isospectral (for the normalized spectrum) to the flat annulus $\Gamma_M=[-M,M]\times \mathbb S^1$. Hence we can write
$$
\sigma_2^*(\nu)=\max_M \bar\sigma_2(\Gamma_M,\nu).
$$
We have seen in Proposition \ref{max_RIS} that
$$
\sigma^*_2(\nu)=\bar\sigma_2(\Sigma(M^*(\nu)),\nu),
$$
where $M^*=M^*(\nu)$ is the unique root of the equation $
(1-\nu)\tanh((1-\nu)x)=\nu\coth( \nu x)$. We summarize these facts in the following theorem:

\begin{theorem}\label{thm_RIS} Let $\Sigma$ be a rotationally symmetric, invariant annulus. Then, for all $\in (0,\frac 12)$, one has:
$$
\bar\sigma_2(\Sigma,\nu)\leq 4\pi\nu\coth( M^*(\nu)),
$$
Equality holds if and only if $\Sigma$ has conformal modulus $M^*(\nu)$.
\end{theorem}

 We prove in Appendix \ref{app:5} the two following fundamental lemmas which describe the behavior of $M^*(\nu)$ and $\sigma_2^*(\nu)$ as functions of $\nu$.

\begin{lemma}\label{M*incr} The function $M^*:(0,\frac 12)\to \mathbb R$ is increasing, and moreover:
$$
\lim_{\nu\to 0^+}M^\star (\nu)=M_0, \quad \lim_{\nu\to\frac 12^-}M^\star (\nu)=+\infty,
$$
where $M_0$ is the unique positive root of the equation $x\tanh(x)=1$.  That is, $M^*$ maps $(0,\frac 12)$ one-to-one to $(M_0,\infty)$.
\end{lemma}
Note that $M_0\approx 1.2$ is the conformal modulus of the critical catenoid.

\begin{lemma}\label{s*incr} The function
$
\sigma^*_2:(0,\frac 12)\to\mathbb R
$
is decreasing, and we have:
$$
\lim_{\nu\to 0^+}\sigma_2^*(\nu)=\frac{4\pi}{M_0}, \quad \lim_{\nu\to\frac 12^-}\sigma_2^*(\nu)=2\pi.
$$
where $M_0$ is the conformal modulus of the critical catenoid. Therefore for all $\nu\in [0,\frac 12]$ (hence, also for all $\nu\in\mathbb R$)  one has
$$
\sigma_2^*(\nu)\leq \frac{4\pi}{M_0},
$$
where $\frac{4\pi}{M_0}$ is the second normalized non-magnetic Steklov eigenvalue (i.e., for $\nu=0$) of the critical catenoid.
\end{lemma}

We are able to prove that Theorem \ref{thm_RIS} holds without the symmetry assumption, namely, it holds for all rotationally invariant annuli. This is the main result of this section.

\begin{thm}\label{max_RI}
Let $\Sigma$ be a rotationally invariant annulus. Then
$$
\sigma_2(\Sigma,\nu)\leq\sigma_2^*(\nu),
$$
where $\sigma_2^*(\nu)=4\pi(1-\nu)\tanh((1-\nu)M^*(\nu))=4\pi\nu\coth(\nu M^*(\nu))$ and $M^*(\nu)$ is the unique positive root of $(1-\nu)\tanh((1-\nu)x)=\nu\coth(\nu x)$. Equality holds if and only if $\Sigma$ is $\sigma$-homothetic to $[-M^*(\nu),M^*(\nu)]\times\mathbb S^1$.
\end{thm}

This result is the analogue of \cite{fs1} for the non-magnetic Steklov case. We prove this theorem in Section \ref{app:5}. As a corollary we have

\begin{cor}\label{cor_max_RI} For any rotationally invariant annulus $\Sigma$ and any flux $\nu\in\mathbb R$ one has the upper bound:
$$
\bar\sigma_2(\Sigma,\nu)\leq \frac{4\pi}{M_0}
$$
with equality if and only if $\nu$ is an integer and $\Sigma$ is $\sigma$-homothetic to the critical catenoid. 
\end{cor}

Hence equality holds if and only if $\nu$ ia an integer and $\Sigma$ admits a free boundary conformal minimal immersion in the unit ball which restricts to a homothety on the boundary. 

\medskip We believe that the upper bound of Theorem \ref{max_RI} holds for any annulus.

\medskip Concerning the case of $\nu=\frac{1}{2}$, we recall from Section \ref{spec:rot} that for rotationally symmetric, invariant annuli, the first eigenvalue is double: $\bar\sigma_1(\Sigma,\nu)=\bar\sigma_2(\Sigma,\nu)=2\pi\tanh(M/2)$. Hence there is no maximiser for $\bar\sigma_2(\Sigma,\nu)$ in the class of rotationally invariant, symmetric annuli, $\sup_{\Sigma}\bar\sigma_2(\Sigma,\nu)=2\pi$, and is reached asymptotically by families of rotationally invariant, symmetric annuli with $M\to+\infty$. In view of Theorem \ref{max_RI}, the same considerations hold if we consider  rotationally invariant annuli.

%


\section{Geometry of maximisers: first eigenvalue}\label{geom1}

Theorem \ref{max_s1} says that among all annuli $\Sigma$ of conformal modulus $M\in (0,\infty)$, $\bar\sigma_1(\Sigma,\nu)$ has a maximum: $\bar\sigma_1(\Sigma,\nu)\leq 4\pi\nu\tanh(\nu M)$. A maximiser is the cylinder $\Gamma_M=[-M,M]\times\mathbb S^1$, and any annulus $\sigma$-homothetic to $\Gamma_M$. In this section we consider some distinguished representatives of the maximiser in the conformal class and study their geometric properties. Inspired by the work of Fraser and Schoen \cite{fs1} we ask the following question:  {\it is there a noteworthy class of free boundary rotation surfaces which realize the equality? Are these surfaces minimal  in some sense?}

\medskip

We will think of $\Sigma$ as being obtained by rotating the graph of the function $x=\rho(z)$ around the $z$-axis over an interval $[-L,L]$  and we assume $\rho(z)> 0$ for all $z\in (-L,L)$.  Hence a parametrization can be:
$$
f(t,\theta)=(\rho(t)\cos\theta,\rho(t)\sin\theta, t), \ \ \  (t,\theta)\in [-L,L]\times\mathbb S^1.
$$
Note that the function $\rho$ measures the distance to the axis of rotation (the $z$-axis). Through all the paper, we use $\rho$ to denote the radial coordinate in $\mathbb R^3$ for the cylindrical coordinates $(\rho,\theta,z)$. In Cartesian coordinates $(x,y,z)$, $\rho=\sqrt{x^2+y^2}$.

\begin{defi}[$\alpha$-surface]
Let $\alpha\in\mathbb R$ be a real parameter. 
We say that the rotation surface $\Sigma$ is an {\it $\alpha$-surface} if the profile $\rho$ satisfies
\begin{equation}\label{asurf-ODE}
\rho\rho''=\alpha(1+\rho'^2).
\end{equation}
\end{defi}
These surfaces are examples of the so-called {\it Weingarten surfaces}, that is, surfaces whose principal curvatures $\kappa_1$ and  $\kappa_2$ satisfy a certain functional relation $\Phi(\kappa_1,\kappa_2)=0$. In fact, it turns out that the principal curvatures of $\alpha$-surfaces are proportional (see Theorem \ref{geom-asurf} here below). These surfaces have been considered in \cite{collharr}.  More recently,  the complete classification of rotation surfaces in $\mathbb R^3$ whose principal curvatures satisfy a linear relation has been obtained in \cite{LoPa}.

\medskip

Up to homotheties in $\mathbb R^3$, there is a unique $\alpha$-surface. As a reference $\alpha$-surface we take the one whose profile solves the problem: 
 $$
\begin{cases}
\rho\rho''=\alpha(1+\rho'^2) & {\rm in\ }(-L,L)\\
\rho(0)=1\,,\rho'(0)=0.
\end{cases}
$$
Here $L$ can be also $+\infty$. This surface is symmetric about the $xy$-plane.

\begin{example} For $\alpha=1$ the solution is $\rho(t)=\cosh t$ hence $1$-surfaces are simply the catenoids. 
\end{example}
\begin{example} For $\alpha=-1$ the solution is $\rho(t)=\sqrt{1-t^2}$, which means that $(-1)$-surfaces are round spheres. 
\end{example}

\begin{figure}[htp]
\centering
\includegraphics[width=0.3\textwidth]{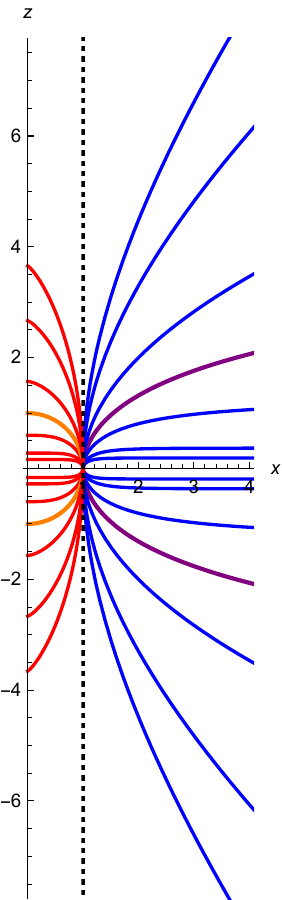}
\caption{Profiles of $\alpha$-surfaces for different values of $\alpha$, in the $xz$ plane, $x>0$. Surfaces with $\alpha>0$ are represented in blue (unbounded), surfaces with $\alpha<0$ are represented in red (bounded); the catenoid ($\alpha=1$) is represented in purple, while the sphere ($\alpha=-1$) is represented in orange; as $\alpha\to 0$ the profile approaches the dotted line $x=1$ (the surface approaches an unbounded flat cylinder); as $\alpha\to\pm\infty$ the profile flattens on the $x$-axis (the surface approaches, respectively, to the $xy$-plane minus the unit disk and to the unit disk).}
\label{fig_wsurf}

\end{figure}

In the next theorem we review the geometric properties of the $\alpha$-surfaces: in particular, note the characterization $c)$ which asserts that the $\alpha$-surface is minimal in a weighted sense.

\begin{theorem}\label{geom-asurf} Let $\Sigma$ be a rotation surface. The following are equivalent:
\begin{enumerate}[a)]
\item $\Sigma$ is an $\alpha$-surface.
\item The principal curvatures $\kappa_1,\kappa_2$ are proportional; precisely, if $\kappa_1$ denotes the meridian curvature (which coincides with the curvature of the profile), then one has at every point
$$
\kappa_1+\alpha \kappa_2=0.
$$
\item $\Sigma$ is critical for the functional ${\rm Area}_{\alpha}(\Sigma):= \int_{\Sigma}\rho^{\alpha-1}\,dv$.
\item The function $u=\rho^{\alpha}$ is a $\Delta_A$-harmonic function on $\Sigma$, where $A=\alpha d\theta$ is the harmonic one-form having flux $\alpha$ around any parallel circle. 
\end{enumerate}
\end{theorem}

The equivalence between a) and b) was proved in \cite{collharr}. The equivalence between a), c) and d) appears to be new.  In particular, the equivalence between a) and c) extends Euler's theorem (stating that the only minimal rotation surface is a catenoidal slab) to rotation surfaces which are stationary for the weighted area ${\rm Area}_{\alpha}(\Sigma)$ (Euler's theorem being the case $\alpha=1$). 

The equivalence between a) and d) has further spectral consequences, see  Section \ref{further} below. 

We prove the theorem in Appendix \ref{app:6}. We note that the function $\rho^{\alpha}$ is $\bar\Delta_A$-harmonic in $\mathbb R^3$, where $\bar\Delta_A$ denotes the magnetic Laplacian in $\mathbb R^3$ and $A$ is the potential one-form in $\mathbb R^3$ given by $A=\frac{\alpha}{\rho^2}(-ydx+xdy)$: in fact, this specific $A$ restricts to $\alpha d\theta$ on $\Sigma$.

\smallskip

Observe also that when $\alpha=1$ the weighted area in c) is the usual area, and in fact $1$-surfaces (i.e. catenoids) are minimal in the classical sense.

\subsection{Critical $\alpha$-surface and the first eigenvalue}

We now restrict to the case $\alpha>0$, so that the profile of the corresponding $\alpha$-surface is convex. We conduct the line from the origin in the $xz$-plane ($x,z>0$) which is tangent to the profile: the contact point $(T,0,L)$ is unique  and finite by the convexity of the profile and by the fact that $(\rho^{-1})'(x):(1,+\infty)\to\mathbb R$ decreases from $+\infty$ to $0$ (see e.g., \cite{collharr}, see also the proof of Lemma \ref{l40}). Therefore the surface obtained by rotation of the graph above the interval $[-L,L]$ is free boundary in the ball of radius $R=\sqrt{T^2+L^2}$.

\begin{defi}[Critical $\alpha$-surface] Let $\alpha,T,L$ be as above, and let $\Sigma$ be the $\alpha$ surface obtained by rotating the graph above the interval $[-L,L]$. Scaling to the unit ball, we get a unique free boundary $\alpha$-surface. We call it the {\it critical $\alpha$-surface}.
\end{defi}

\begin{example} The critical $1$-surface is just the critical catenoid $\Sigma_c$. 
\end{example}

It is not difficult to see that the family of critical $\alpha$-surfaces for $\alpha>0$ foliate the unit $3$-ball minus the segment $S$ joining the two poles, and as $\alpha\to 0$ degenerates to $S$, see Figure \ref{fig_fol0}.

\begin{figure}[htp]
\centering
\includegraphics[width=0.5\textwidth]{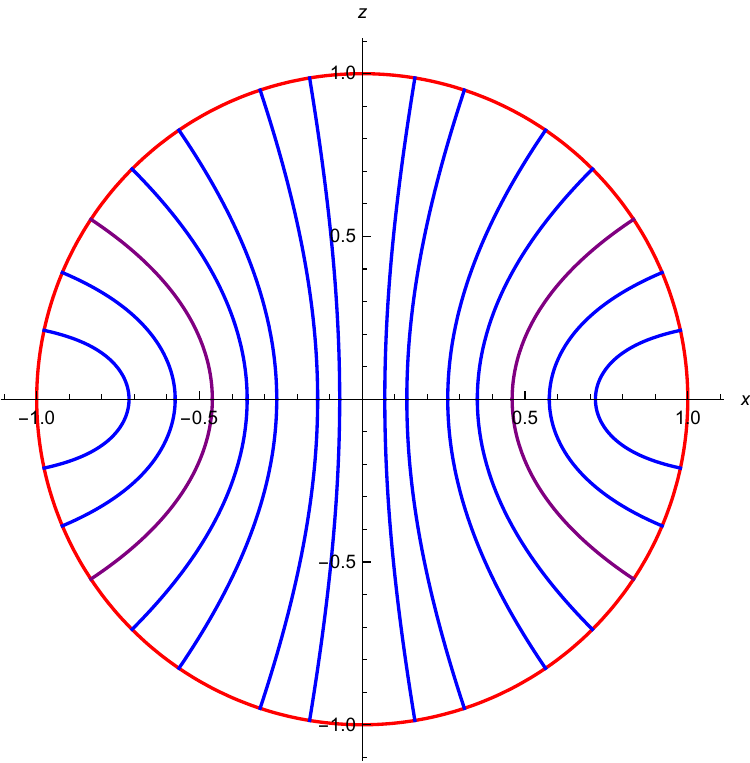}
\caption{Sections on the $xz$-plane of critical $\alpha$-surfaces. In purple the section corresponding the critical catenoid, i.e., $\alpha=1$.}
\label{fig_fol0}
\end{figure}

Let $M(\alpha)$ be the conformal modulus of the critical $\alpha$-surface. We have the following monotonicity property of $M(\alpha)$: 

\begin{prop}\label{Madecr} The function $M:(0,\infty)\to (0,\infty)$ is decreasing from $+\infty$ to $0$, that is:
$$
\lim_{\alpha\to 0^+}M(\alpha)=+\infty, \quad \lim_{\alpha\to+\infty}M(\alpha)=0.
$$
The proof is presented in Appendix \ref{app:6}. Therefore, given $M>0$ there is a unique critical $\alpha$-surface with conformal modulus $M$.
\end{prop} 

As a consequence, if $\Sigma$ is the critical $\alpha$-surface with conformal modulus $M$, then the function $\alpha=\alpha(M)$ satisfies:
$$
\lim_{M\to 0^+}\alpha(M)=+\infty, \quad \lim_{M\to +\infty}\alpha(M)=0.
$$
We deduce that the critical $\alpha$-surfaces are noteworthy maximisers for the first normalized magnetic eigenvalue. 

\begin{theorem}\label{asurf-s1} Let $\Sigma$ be an annulus with conformal modulus $M$. Then, for all $\nu$ one has:
$$
\bar\sigma_1(\Sigma,\nu)\leq \bar\sigma_1(\Sigma(\alpha),\nu),
$$
where $\Sigma(\alpha)$ is the unique critical $\alpha$-surface with conformal modulus $M$. 
\end{theorem}

This is a consequence of Theorem \ref{max_s1}. In fact $\Sigma(\alpha)$ is $\sigma$-homothetic to $\Gamma_M=[-M,M]\times\mathbb S^1$ and hence realizes the equality  $\bar\sigma_1(\Sigma(\alpha),\nu)= \bar\sigma_1(\Gamma_M,\nu)=4\pi\nu\tanh(\nu M)$.

\subsection{A further property of critical $\alpha$-surfaces}\label{further}

The following fact is simple to prove, but plays an important role in the work of Fraser and Schoen \cite{fs1,fs2} on the Steklov problem.

\begin{theorem} Let $\Sigma$ be a surface contained in the unit ball $B\subset\mathbb R^3$ having boundary $\partial\Sigma\subset\partial B$.  Then, $\Sigma$ is a minimal surface meeting the boundary of the ball orthogonally if and only if the coordinate functions are eigenfunctions of the (non-magnetic) Steklov problem on $\Sigma$ associated to the eigenvalue $1$.
\end{theorem}

The following theorem expresses a similar property for critical $\alpha$-surfaces.

\begin{theorem}\label{geom_car_asurf} A rotation surface $\Sigma$ contained in the unit ball $B\subset\mathbb R^3$ such that $\partial\Sigma\subset\partial B$  is a critical $\alpha$-surface if and only if the function $\rho^{\alpha}$ is a $\Delta_A$-Steklov eigenfunction associated to the eigenvalue $\alpha$.  Here $A=\alpha d\theta$.
\end{theorem} 
The proof is presented in Appendix \ref{app:6}.
Again, note that the considered magnetic potential $A$ on $\Sigma$ is the restriction of the one-form $A=\frac{\alpha}{\rho^2}(-ydx+xdy)$ in $\mathbb R^3$, and $\rho^{\alpha}$ is $\bar\Delta_A$-harmonic in $\mathbb R^3$.


\section{Geometry of maximisers: second eigenvalue} \label{geom2}

Theorem \ref{max_RI} states that among all rotationally invariant annuli $\Sigma$, the second normalized eigenvalue $\bar\sigma_2(\Sigma,\nu)$ has a maximum: $\sigma_2^*(\nu)=4\pi(1-\nu)\tanh((1-\nu)M^*)=4\pi\nu\coth(\nu M^*)$ where $M^*=M^*(\nu)$ is the unique positive solution to $(1-\nu)\tanh((1-\nu)x)=\nu\coth(\nu x)$. We now discuss maximisers for this object. This should be considered as being the natural generalization of the results of Fraser and Schoen \cite{fs1} in the non-magnetic case. We assume $\nu\in (0,\frac 12)$: the case $\nu=0$ has been already treated in \cite{fs1}, while in the case $\nu=\frac 12$ we have already seen that there are no maximisers in the class or rotationally invariant annuli.

\medskip

We recall that the second eigenvalue $\sigma_2(\Sigma,\nu)$ of the cylinder $\Gamma_{M^*}=[-M^*,M^*]\times\mathbb S^1$ having conformal modulus $M^*=M^*(\nu)$  has multiplicity two, with associated eigenspace being spanned by the eigenfunctions, in coordinates $(t,\theta)\in[-M^*,M^*]\times\mathbb S^1$,
$$
u_1(t,\theta)=\cosh((1-\nu)t) e^{i\theta}, \quad u_2(t,\theta)=\sinh(\nu t)
$$
(recall that $\Gamma_{M^*}$ is in fact a maximiser of $\sigma_2(\Sigma,\nu)$).
The first eigenfunction is complex-valued, the second one is real-valued. We normalize $u_1$ and $u_2$ so that the embedding
$F=(u_1,u_2):\Gamma_{M^*}\to\mathbb C\times\mathbb R\subset\mathbb C\times\mathbb C$ maps $\bd\Sigma$ to $\bd B$, where $B$ is the unit ball in $\mathbb C\times\mathbb R$, which we naturally identify with $\mathbb R^3$.  This is accomplished by taking
\begin{equation}\label{aimmersion}
\twosystem
{u_1(t,\theta)=\frac{a}{1-\nu}\cosh((1-\nu)t)e^{i\theta}}
{u_2(t,\theta)=\frac{a}{\nu}\sinh(\nu t)}
\end{equation}
where $a=a(\nu)$ is given by
$$
a=\left(\frac{\cosh^2((1-\nu)M^*)}{(1-\nu)^2}+\frac{\sinh^2(\nu)M^*)}{\nu^2}\right)^{-\frac 12}.
$$
Having done that,
we have to choose the metric $g=g_{\nu}$ on the surface $\Gamma_{M^*}$.

To that end, let $h(\cdot,\cdot)$ denote the Hermitian inner product on $\mathbb C\times\mathbb C$; identifying $\mathbb C\times\mathbb C$ with $\mathbb R^4$, we see that  the Euclidean inner product is then $\scal{X}{Y}=\mathcal R( h(X,Y))$. Given the smooth embedding $F=(F_1,F_2):\Gamma_{M^*}\to \mathbb C\times\mathbb C$, we now pull-back the Euclidean metric using the magnetic differential $d^A$ instead of the usual differential $d$. Explicitly, we define the metric $g_{\nu}$ on $\Gamma_{M^*}$ as follows:
\begin{equation}\label{newmetric}
g_{\nu}(X,Y)=\Re (h(d^A F(X),d^AF(X))
\end{equation}
where $A=\frac{\nu}{x^2+y^2}(-ydx+xdy)$ is the potential form with flux $\nu$, as usual. Here we are using coordinates $(x+iy,z)\in\mathbb C\times\mathbb R\subset\mathbb C\times\mathbb C$. 
Summarizing, we have the following properties:
\color{black}

\begin{theorem}\label{FSs2} Let $\Gamma_{M^*}=[-M^*(\nu),M^*(\nu)]\times \mathbb S^1$ with the metric $g_{\nu}$ defined in \eqref{newmetric}, and let $F=(u_1,u_2)$ be the immersion of $\Gamma_{M^*}$ into $\mathbb C\times\mathbb R$ defined by \eqref{aimmersion}. Then, for all $\nu\in(0,\frac 12)$:

\begin{enumerate}[a)]
\item  $g_{\nu}$ is conformal to the flat metric; precisely, $g_{\nu}=\psi(t)(dt^2+d\theta^2)$
where 
$$
\psi(t)=a^2\Big(\sinh^2((1-\nu)t)+\cosh^2(\nu t)\Big)=a^2\Big(\cosh^2((1-\nu)t)+\sinh^2(\nu t)\Big)
$$
\item The immersion $F:\Gamma_{M^*}\to\mathbb C\times\mathbb R$ is free boundary in the unit ball.
\item For all rotationally invariant annuli $\Sigma$ one has 
$$
\bar\sigma_2(\Sigma,\nu)\leq \bar\sigma_2(\Gamma_{M^*},\nu)\leq\bar\sigma_2(\Sigma_c,0)=\frac{4\pi}{M_0}
$$
where $\Sigma_c$ is the critical catenoid (with conformal modulus $M_0\approx 1.2$). 
\end{enumerate}
\end{theorem}

Since the metric  $g_{\nu}$ is conformal to the flat one, we get that $\Delta_AF=0$. Recall that for an immersion $F:\Sigma\to\mathbb R^n$, one always has $\Delta F=H$, where $H$ is the mean curvature vector. Then, the following definition makes sense.

\begin{defi}[$A$-minimality] An immersion $F:\Sigma\to\mathbb C\times\mathbb R$ is called {\it $A$-minimal} if $\Delta_AF=0$, where the metric on $\Sigma$ (see \eqref{newmetric}) is the pull-back of the Euclidean metric on $\mathbb C\times \mathbb R$ by the map $F$ using $d^A$.
\end{defi}

We observe that the following result holds, which is the exact analogous of the result proved in \cite{fs1} for the classical Steklov problem.

\begin{theorem}\label{FSmagnetic} Let $\Sigma$ be a rotation surface such that $\Sigma\subset B$, where $B$ is the unit ball in $\mathbb C\times\mathbb R$, and $\partial\Sigma\subset\partial B$. Then, $\Sigma$ is $A$-minimal and meets the boundary of $B$ orthogonally if and only if the coordinate functions of the immersion are $\Delta_A$-Steklov eigenfunctions associated to the eigenvalue $\sigma=1$.
\end{theorem}

The proofs of Theorems \ref{FSs2} and \ref{FSmagnetic} are presented in Appendix \ref{app:7}.

\appendix

\section{Proofs of Section \ref{sec:mag:lap}}\label{app:2}

In this section we prove Proposition \ref{sigmah}. Before showing the proof, we need to recall some preliminaries.

\medskip

Let $(\Sigma,g)$ and $(\Gamma,h)$ be two Riemannian surfaces which admit a conformal diffeomorphism $\phi:(\Sigma,g)\to(\Gamma,h)$ so that $\phi^*h=\psi g$ for a smooth function $\psi$ on $
\Sigma$. It is well-known that if $\Delta_g$ denotes the usual Laplacian associated to the metric $g$ one has, for all smooth functions $u$ on $\Gamma$,  $\Delta_g\phi^*u=\psi\phi^*\Delta_hu$. Here and in what follows $\phi^*u:=u\circ\phi$. In particular, $\phi^*$ takes $\Delta_h$-harmonic functions to $\Delta_g$-harmonic functions. The next lemma states that this is true also in the case of the magnetic Laplacian. We will use the notation $\Delta_{A,h}$ for the magnetic Laplacian on $\Gamma$ with potential one-form $A$ (analogous notation applies for $\Sigma$). We shall denote by $\phi^*A$  the pull-back of a one-form $A$ on $\Gamma$ through $\phi$.

\begin{lemma}\label{lemma_conf} Let $\phi:(\Sigma,g)\to (\Gamma, h)$ be a conformal diffeomorphism between Riemannian surfaces with $\phi^*h=\psi g$. For any one-form $A$ on $\Gamma$ we have:
\begin{enumerate}[a)]
\item
$
\Delta_{\phi^*A,g}(\phi^*u)=\psi\phi^*(\Delta_{A,h}u)
$
\item
$
d^{\phi^*A}\phi^*u=\phi^*d^Au
$
\item If $N_{\Sigma}$ and $N_{\Gamma}$ are the outer unit normals, then $d\phi(N_{\Sigma}(x))=\sqrt{\psi(x)}N_{\Gamma}(\phi(x))$ for all $x\in\bd\Sigma$.
\item If $\phi$ is a $\sigma$-homothety, that is, $\psi\equiv c$ on $\partial\Sigma$ (i.e., it is constant on $\partial\Sigma$), then
$
\abs{\bd\Gamma}=\sqrt{c}\abs{\bd\Sigma}.
$
\end{enumerate}
\end{lemma} 

We are now ready to prove Proposition \ref{sigmah}.

\begin{proof}[Proof of Proposition \ref{sigmah}]
Let $(\Sigma,g)$, $(\Gamma,h)$ be two Riemannian surfaces which are $\sigma$-homothetic, i.e., there exists a conformal diffeomorphism $\phi:(\Sigma,g)\to(\Gamma,h)$ which restricts to an homothety on $\partial\Sigma$. Let $N_{\Gamma},N_{\Sigma}$ denote the outer unit normals to $\Gamma,\Sigma$, respectively. Let $A$ be a one-form on $\Gamma$.

\medskip

Let $u$ be a Steklov eigenfunction of $\Delta_A$ on $\Gamma$ associated to the eigenvalue $\sigma_k$. This means: 
$$
\begin{cases}
\Delta_Au=0 & {\rm in\ }\Gamma\\
d^Au(N_{\Gamma})=\sigma_k u & {\rm on\ }\partial\Gamma.
\end{cases}
$$
From point $a)$ of Lemma \ref{lemma_conf}, we deduce that the function $\phi^*u$ is $\Delta_{\phi^*A}$-harmonic on $\Sigma$. Now, by points $b)$ and $c)$ of Lemma \ref{lemma_conf} we get
$$
d^{\phi^*A}(\phi^*u)(N_{\Sigma})=\phi^*d^Au(N_{\Sigma})
=d^Au(d\phi(N_{\Sigma}))
=\sqrt{\psi} d^Au(N_{\Gamma}\circ\phi)
=\sqrt{\psi}\sigma_k \phi^*u.
$$
By assumption, $\sqrt{\psi}$ is constant on the boundary, hence $\sigma'_k:=\sqrt{\psi}\sigma_k$ is a Steklov eigenvalue of $\Delta_{\phi^*A}$ on $\Sigma$. Then, $\phi^*$ induces an isomorphism on each eigenspace, and normalized eigenvalues are preserved, as follows from point $d)$ of Lemma \ref{lemma_conf}:
\begin{equation}\label{ns}
\sigma'_k\abs{\bd\Sigma}=\sqrt{\psi}|\partial\Sigma|\sigma_k=|\partial\Gamma|\sigma_k.
\end{equation} 
We are now ready to complete the proof. Assume now that $\Gamma,\Sigma$ are Riemannian annuli. If $A$ is a closed one-form having flux $\nu$ on $\Gamma$, then $\phi^*A$ is a closed one-form with  flux $\pm\nu$ on $\Sigma$, depending on the orientation; this implies that the the spectrum of the pair $(\Sigma,\phi^*A)$ is just what we denoted $\sigma_k(\Sigma,\nu)$ (remember that the sign of the flux has no effect on the spectrum). 
Then, by \eqref{ns}, $\Sigma$ and $\Gamma$ have the same normalized Steklov spectrum, for all fluxes $\nu$:
$$
\bar\sigma_k(\Sigma, \nu)=\bar\sigma_k(\Gamma,\nu).
$$

\end{proof}

\section{Proofs of Section \ref{sec:bound:first} }\label{app:4}

In this Appendix we prove Theorems \ref{max_s1}, \ref{hear_conf} and \ref{planar}.

\begin{proof}[Proof of Theorem \ref{max_s1}]
If $\nu=0$ (or $\nu\in\mathbb Z$) the first eigenvalue is zero and the result is straightforward. Let then $\nu\in(0,\frac 12]$. By assumption there is a conformal diffeomeorphism 
$
\phi:\Sigma\to \Gamma_M
$
where $\Gamma_M$ is the flat cylinder $[-M,M]\times\mathbb S^1$ with coordinates $(t,\theta)$.  Consider the first eigenfunction corresponding to $\sigma_1(\Gamma_M,\nu)$, which is 
$
u(t)=\cosh(\nu t),
$
and which is real and constant on $\partial\Gamma_M$. Recall that we are considering the potential one-form $A=\nu d\theta$ on $\Gamma_M$. We take as test function for the eigenvalue $\sigma_1(\Sigma,\phi^*A)$ (the first magnetic Steklov eigenvalue with potential $\phi^*A$) the pull-back of $u$ by $\phi$, namely, $\phi^*u:=u\circ\phi$. Then, by the min-max principle \eqref{minmax}:
$$
\sigma_1(\Sigma,\phi^*A)\int_{\partial\Sigma}(\phi^*u)^2ds\leq \int_{\Sigma}
|d^{\phi^*A}(\phi^*u)|^2dv.
$$

Now, the flux is invariant by diffeomorphisms, hence 
$\phi^*A$ has flux $\nu$ around any of the boundary components of $\Sigma$, and moreover $\phi^*A$ is closed (see Lemma \ref{lemma_conf}, $a)$). This means that, by gauge invariance $\sigma_1(\Sigma,\phi^*A)=\sigma_1(\Sigma,\nu)$. By the conformal invariance of the magnetic energy:
\begin{multline*}
\int_{\Sigma}
|d^{\phi^*A}(\phi^*u)|^2dv=\int_{\Gamma_M}|d^Au|^2\,dv
=\sigma_1(\Gamma_M,\nu)\int_{\bd\Gamma_M}u^2\\
=4\pi u(M)^2\sigma_1(\Gamma_M,\nu)
=4\pi u(M)^2\nu\tanh(\nu M).
\end{multline*}
Finally, $\phi^*u|_{\partial\Sigma}=u(M)$ is constant on $\partial\Sigma$, so that $\int_{\partial\Sigma}(\phi^*u)^2ds=|\partial\Sigma|u(M)^2$.
We conclude:
$$
|\partial\Sigma|\sigma_1(\Sigma,\nu)\leq 4\pi\nu\tanh( \nu M),
$$
which is the claimed inequality. 

\medskip

Now assume that equality holds. Let $h$ be the metric on $\Sigma$, and $g$ that on $\Gamma_M$. Since $\phi$ is conformal, we have $\phi^*g=\psi h$ for some smooth $\psi$. It is enough to show that the conformal factor $\psi$ is constant on $\bd\Sigma$. 

\smallskip 

Now, if equality holds, then $\phi^*u$ must be an eigenfunction, and therefore
$$
d(\phi^*u)(N)=\sigma_1 \phi^*u
$$ 
on the boundary of $\Sigma$. The right hand side is constant on $\bd\Sigma$ hence
$$
du(d\phi_x(N))=c
$$
for all $x\in\bd \Sigma$. One has (recall that $N$ is the unit normal vector to $\partial\Sigma$):
$$
g(d\phi_x(N), d\phi_x(N))=\phi^*g_x(N,N)=\psi (x)h(N,N)=\psi (x)
$$
and therefore, since $d\phi_x(N)$ is collinear with the unit normal $N_{\Gamma}$ to $\bd\Gamma$, we have
$$
d\phi_x(N)=\dfrac{1}{\sqrt{\psi(x)}}N_{\Gamma},
$$
so that
$$
c=du(d\phi_x(N))=\dfrac{1}{\sqrt{\psi(x)}}du(N_{\Gamma})=\dfrac{1}{\sqrt{\psi(x)}}\sigma_1(\Gamma,\nu)u(M)=\dfrac{c'}{\sqrt{\psi(x)}},
$$
where $c'$ is a constant: then, $\psi(x)$ is also a constant on $\bd\Sigma$. The proof is complete. 
\end{proof}

\begin{proof}[Proof of Theorem \ref{hear_conf}]
We have already proved in Theorem \ref{max_s1} an upper bound:
$$
\bar\sigma_1(\Sigma,\nu)\leq 4\pi\nu\tanh(\nu M(\Sigma))\leq 4\pi\nu^2 M(\Sigma)
$$
for all $\nu\in[0,\frac 12]$. We show now that $\bar\sigma_1(\Sigma,\nu)\geq 4\pi M(\Sigma)\nu^2+o(\nu^2)$ as $\nu\to 0$. This, along with the upper bound, proves the result.

\medskip

We can assume $\nu\ne 0$ (otherwise the first eigenvalue is zero with constant eigenfunctions). Let $u_{\nu}$ denote a first Steklov eigenfunction  on $\Sigma$ normalized by $\int_{\partial\Sigma}|u_{\nu}|^2ds=1$, for all $\nu\in (0,\frac 12]$. For any $u\in H^1(\Sigma)$ we have (recall that $A$ is a real potential one-form)
$$
2|u|d |u|=d|u|^2=2\Re(\bar u\,d u)=2\Re(\bar u(d u-iAu))=2\Re(\bar u\,d^Au)
$$
and since $|2\Re(\bar u\,d^Au)|\leq 2|u||d^Au|$ we deduce the diamagnetic inequality $|d |u||\leq|d^Au|$. In particular, the family $\{|u_{\nu}|\}_{\nu\in(0,1/2)}$ is uniformly bounded in $H^1(\Sigma)$ endowed with the equivalent norm $\|d u\|_{L^2(\Sigma)}+\| u\|_{L^2(\partial\Sigma)}$. Moreover, $\|d|u_{\nu}|\|\leq\|d^Au_{\nu}\|\to 0$ as $\nu\to 0^+$. Therefore we deduce that $|u_{\nu}|\rightarrow c$, strongly in $H^1(\Sigma)$, and therefore also in $L^2(\partial\Sigma)$. From the normalization on $\|u_{\nu}\|_{L^2(\partial\Sigma)}$ we deduce that $c=\frac{1}{\sqrt{|\partial\Sigma|}}$.

\medskip

Now, let $\phi:\Gamma_M\to\Sigma$ be a conformal diffeomorphism, where $\Gamma_M=[-M,M]\times\mathbb S^1$ and $M=M(\Sigma)$ is the conformal modulus of $\Sigma$.  Consider the test function $\phi^*u_{\nu}:=u_{\nu}\circ\phi$ for the magnetic Steklov problem on $\Gamma_M$, where we choose as potential one-form the form $\phi^*A$: it is closed and has flux $\nu$ (see also Lemma \ref{lemma_conf}  and Proposition \ref{sigmah} in Appendix \ref{app:2}). From the min-max principle \eqref{minmax} we have
$$
\sigma_1(\Gamma_M,\nu)\leq\frac{\int_{C_M}|d^{\phi^*A}\phi^*u_{\nu}|^2 dv}{\int_{\partial \Gamma_M}|\phi^*u_{\nu}|^2ds}=\frac{\int_{\Sigma}|d^Au_{\nu}|^2dv}{\int_{\partial \Gamma_M}|\phi^* u_{\nu}|^2ds}=\frac{\sigma_1(\Sigma,\nu)}{\int_{\partial \Gamma_M}|\phi^* u_{\nu}|^2ds}.
$$
The first equality follows from the conformal invariance of the magnetic energy (see Subsection \ref{conf:inv}). Therefore we have
\begin{multline*}
\sigma_1(\Sigma,\nu)\geq\sigma_1(\Gamma_M,\nu)\int_{\partial \Gamma_M}|\phi^* u_{\nu}|^2ds=\frac{4\pi\sigma_1(\Gamma_M,\nu)}{|\partial\Sigma|}+\sigma_1(\Gamma_M,\nu)\left(\int_{\partial \Gamma_M}|\phi^* u_{\nu}|^2ds-\frac{4\pi}{|\partial\Sigma|}\right)\\
=\frac{4\pi\nu\tanh(\nu M)}{|\partial\Sigma|}+\nu\tanh(\nu M)\left(\int_{\partial \Gamma_M}|\phi^* u_{\nu}|^2ds-\frac{4\pi}{|\partial\Sigma|}\right)\\=\frac{4\pi M \nu^2}{|\partial\Sigma|}+\nu\tanh(\nu M)\left(\int_{\partial \Gamma_M}|\phi^* u_{\nu}|^2ds-\frac{4\pi}{|\partial\Sigma|}\right)+o(\nu^2).
\end{multline*}
The proof is complete once we show that
$$
\lim_{\nu\to 0^+}\left(\int_{\partial \Gamma_M}|\phi^* u_{\nu}|^2ds-\frac{4\pi}{|\partial \Sigma|}\right)=0.
$$
This is proved by recalling that $|u_{\nu}|\to c=\frac{1}{\sqrt{|\partial\Sigma|}}$ in $L^2(\partial\Sigma)$, and therefore
$$
\int_{\partial \Gamma_M}|\phi^* u_{\nu}|^2ds=\int_{\partial\Sigma}|u_{\nu}|^2|d\phi^{-1}|ds\to c^2\int_{\partial\Sigma}|d\phi^{-1}|ds=\frac{|\partial \Gamma_M|}{|\partial\Sigma|}=\frac{4\pi}{|\partial\Sigma|}.
$$

\end{proof}

In order to prove Theorem \ref{planar} we need to recall the following auxiliary result:

\begin{lemma}\label{lemma_bucur}
Let $\Omega$ be a smooth domain in $\mathbb R^n$ and let $w\in L^{\infty}(\partial\Omega)$, $w\geq 1$ be a density on $\partial\Omega$. Then there exists a sequence of domains $\Omega_{\eps}$, $\eps\in(0,\eps_0)$, that converges uniformly to $\Omega$ and such that $\mathcal H|_{\partial\Omega_{\eps}}\rightharpoonup w\mathcal H|_{\partial\Omega}$
weakly* in the sense of measures, where $\mathcal H$ is the $n-1$ dimensional Hausdorff measure. In particular,
$$
|\partial\Omega_{\eps}|\to\int_{\partial\Omega}w\, ds
$$
and
$$
\sigma_k(\Omega_{\eps},\nu)\to\sigma_k^w(\Omega,\nu),
$$
where $\sigma_k^w(\Omega,\nu)$ are the eigenvalues of the weighted Steklov problem \eqref{weigh_AB} on $\Omega$.
\end{lemma}
One can find a detailed proof e.g., in \cite{bucur,FL_Steklov} for the Laplace operator. The adaptation to the case of the magnetic Laplacian is straightforward.

\begin{proof}[Proof of Theorem \ref{planar}]
Let $A(r_0):=\{x\in\mathbb R^2:r_0<|x|<1\}\subset\mathbb R^2$, with $0<r_0<1$ be a planar annulus, and let $A=\nu d\theta$ (in polar coordinates $(r,\theta)$ in $\mathbb R^2$) be a potential one-form on $\mathbb R^2$ (it is closed with flux $\nu$ on each component of $\partial A(r_0)$). Consider the following weighted Steklov problem

\begin{equation}\label{weigh_AB}
\begin{cases}
\Delta_Au=0 & {\rm in\ } A(r_0)\\
d^Au(N)=\sigma w u  & {\rm on\ }\partial A(r_0),
\end{cases}
\end{equation}
where $w\equiv 1$ at $|x|=1$ and $w\equiv c\geq 1$ at $|x|=r_0$ is a positive weight on $\partial A(r_0)$. We denote the eigenvalues by
$$
\sigma_k^w(A(r_0),\nu).
$$
Now we choose $c=\frac{1}{r_0}$, and solve problem \eqref{weigh_AB}. 
 It is  standard to verify, by separation of variables,  that a first eigenfunction is  real, radial, and of the form, in polar coordinates $(r,\theta)$, $u(r)=ar^{\nu}+br^{-\nu}$.  A corresponding first eigenvalue is determined  by the relations:
$$
\begin{cases}
u'(1)=\sigma u(1),\\
-u'(r_0)=\frac{\sigma}{r_0}u(r_0).
\end{cases}
$$
This corresponds to a system of two linear equations in two unknowns $a,b$ which has a non-trivial solution if and only if the determinant of the associated matrix vanished. Such matrix is
$$
\begin{pmatrix}
\nu-\sigma & -\nu-\sigma\\
r_0^{-1+\nu}(\sigma+\nu) & -\nu r_0^{-1-\nu}(\sigma-\nu)
\end{pmatrix}
$$
Imposing the vanishing of the determinant we obtain the two solutions
$$
\sigma_-=\nu\frac{1-r_0^{\nu}}{1+r_0^{\nu}}\,,\ \ \ \sigma_+=\nu\frac{1+r_0^{\nu}}{1-r_0^{\nu}}
$$
and, consequently,
$$
\sigma_1^w(A(r_0),\nu)=\sigma_-=\nu\frac{1-r_0^{\nu}}{1+r_0^{\nu}}.
$$
We apply now Lemma \ref{lemma_bucur} to $A(r_0)$ with weight $w\equiv\frac{1}{r_0}$ on the inner boundary and $w\equiv 1$ on the outer boundary. We deduce that there exists a sequence of domains $\Omega_{\eps}(r_0)$, $\eps\in(0,\eps_0)$,  such that
$$
\lim_{\eps\to 0}|\partial\Omega_{\eps}(r_0)|=\int_{A(r_0)}w=4\pi
$$
and
$$
\lim_{\eps\to 0}\sigma_1(\Omega_{\eps}(r_0),\nu)=\sigma_1^w(A(r_0),\nu)=\nu\frac{1-r_0^{\nu}}{1+r_0^{\nu}},
$$
and, consequently,
$$
\lim_{\eps\to 0}\bar\sigma_1(\Omega_{\eps}(r_0),\nu)=4\pi\nu\frac{1-r_0^{\nu}}{1+r_0^{\nu}}.
$$
By taking $r_0\to 0^+$, and by a diagonal argument, we find a sequence of domains $\Omega_{\eps}$ such that
$$
\lim_{\eps\to 0}\bar\sigma_1(\Omega_{\eps},\nu)=4\pi\nu.
$$
One may think of these domains $\Omega_{\eps}(r_0)$ as deformed circular, planar annuli, whose inner boundary is described by the graph of a rapidly oscillating function over $r_0\mathbb S^1$, with a suitable ratio of the heights and the frequencies of the oscillations (see Figure \ref{fig_planar}).

\end{proof}

\section{Proofs of Section \ref{sec:bound:second}}\label{app:5}

In this section we prove Lemmas \ref{M*incr}, \ref{s*incr} and Theorem \ref{max_RI}.

\begin{proof}[Proof of Lemma \ref{M*incr}]
We recall that $M^*(\nu)$ is the unique positive solution of the following implicit equation in the unknown $M$:
\begin{equation}\label{implicitM0}
F(\nu,M):=(1-\nu)\tanh((1-\nu)M)-\nu\coth(\nu M)=0.
\end{equation}
Note that
$$
\frac{\partial F}{\partial M}=\frac{(1-\nu)^2}{\cosh^2((1-\nu)M)}+\frac{\nu^2}{\sinh^2(\nu M)}>0
$$
for all $M>0$ and all $\nu\in(0,\frac 12)$. Hence we are in the hypothesis of the Implicit Function Theorem: we have a well-defined function $M^*(\nu)$. In particular,
\begin{multline}\label{Mstar_der}
\frac{d}{d\nu}M^*(\nu)=-\left(\frac{\partial F}{\partial M}\right)^{-1}\frac{\partial F}{\partial\nu}_{|_{M=M^*(\nu)}}\\=\left(\frac{(1-\nu)^2}{\cosh^2((1-\nu)M^*(\nu))}+\frac{\nu^2}{\sinh^2(\nu M^*(\nu))}\right)^{-1}\\
\cdot\left(\tanh((1-\nu)M^*(\nu))+\frac{(1-\nu)M^*(\nu)}{\cosh^2((1-\nu)M^*(\nu))}+\coth(\nu M^*(\nu))-\frac{\nu M^*(\nu)}{\sinh^2(\nu M^*(\nu))}\right).
\end{multline}
The functions $\tanh((1-\nu)M)+\frac{(1-\nu)M}{\cosh^2((1-\nu)M)}$ and $\coth(\nu M)-\frac{\nu M}{\sinh^2(\nu M)}$ of the variables $\nu,M$, are positive for any value of $M>0$, $\nu\in(0,\frac 12)$. It is straightforward for the former. As for the latter function, consider $f(x)=\coth(x)-\frac{x}{\sinh^2(x)}=\frac{1}{\sinh^2(x)}\left(\cosh(x)\sinh(x)-x\right)$. The function $\cosh(x)\sinh(x)-x$ has limit $0$ as $x\to 0^+$ and moreover $\frac{d}{dx}\left(\cosh(x)\sinh(x)-x\right)=2\sinh^2(x)\geq 0$. Hence $M^*(\nu)$ has positive derivative, and this concludes the proof of the monotonicity. 

\medskip

We consider now the limits. We first prove that $\lim_{\nu\to\frac 12^-}M^*(\nu)=+\infty$. Assume by contradiction that $M^*(\nu)\leq C$ for some $C\in\mathbb R$ and for all $\nu\in(0,\frac 12)$. Recall that $M^*(\nu)$ is the unique positive root of \eqref{implicitM0}. Using hyperbolic trigonometric identities we also see that $M^*(\nu)$ is the unique positive root of
$$
\cosh((1-2\nu)M)=(1-2\nu)\cosh(M)
$$
in the unknown $M$. Since $\cosh((1-2\nu)x)\geq 1$:
$$
1\leq (1-2\nu)\cosh (C),
$$
which is clearly impossible for $\nu\to\frac 12^-$.

\medskip We prove now that   $\lim_{\nu\to 0^+}M^*(\nu)=M_0$, where $M_0$ is the unique positive root of $M\tanh(M)=1$. Take a sequence $M^*(\nu_n)$, $n\in\mathbb N$, $\nu_n\to 0^-$ as $n\to+\infty$. We have that $M^*(\nu_n)>0$ for all $n\in\mathbb N$. Hence there exists $M^*_{\infty}\geq 0$ such that $M^*(\nu_n)\to M^*_{\infty}$ as $n\to\infty$. Clearly $M^*_{\infty}\ne 0$, otherwise we would get a contradiction as in the previous case $\nu\to\frac 12^-$. Plugging $M^*(\nu_n)$ into \eqref{implicitM0}, and passing to the limit, we get that $M^*_{\infty}$ is positive and solves $M\tanh(M)=1$, hence $M^*_{\infty}=M_0$. This is true for all sequences $M^*(\nu_n)$ with $\nu_n\to 0^+$. Hence the function $M^*(\nu)$ has limit $M_0$ as $\nu\to 0^+$.
\end{proof}

\begin{proof}[Proof of Lemma \ref{s*incr}] 
We prove first that the function $\sigma^*(\nu)$ is decreasing with respect to $\nu$ for $\nu\in(0,\frac 12)$.
We recall that
\begin{equation}\label{sigmastar}
\sigma^*(\nu)=4\pi\nu\coth(\nu M^*(\nu)),
\end{equation}
where $M^*(\nu)$ is the unique positive root of \eqref{implicitM0} in the unknown $M$. Deriving \eqref{implicitM0} we find an expression for $\frac{d}{d\nu}M^*(\nu)$, namely \eqref{Mstar_der}, which we substitute in the formula for the derivative of $\sigma^*(\nu)$; then, we can make the computation by deriving the right-hand side of \eqref{sigmastar}. Altogether we obtain
\begin{equation}\label{der0}
\frac{d}{d\nu}\sigma^*(\nu)=\frac{\sinh(\nu M)\cosh(\nu M)-\nu\left(M(1-\nu)+\nu\frac{\cosh(M)\cosh((1-\nu)M)}{\sinh(\nu M)}\right)}{\sinh^2(\nu M)+\nu^2\cosh^2((1-\nu)M)}|_{|_{M=M^*(\nu)}}.
\end{equation}
Next, we operate the substitution $M^*(\nu)=:\frac{\log(y)}{\nu}$ so that $y>1$. We find that
$$
\frac{d}{d\nu}\sigma^*(\nu)=f_{\nu}(y):=\frac{y^4-1}{(y^2-1)^2}-\frac{4y^2\nu\left(\frac{1-\nu}{\nu}\log(y)+\frac{\nu(1+y^2+y^{2-2/\nu}+y^{2/\nu})}{2(y^2-1)}\right)}{(y^2-1)^2(1-\nu)^2+(y^2+y^{2/\nu})^2\nu^2y^{-2/\nu}}.
$$
We set $g_{\nu}(y):=\left((y^2-1)^2(1-\nu)^2+(y^2+y^{2/\nu})^2\nu^2y^{-2/\nu}\right)f_{\nu}(y)$. Note that $f_{\nu}(y)<0$ if and only if $g_{\nu}(y)<0$. Hence, the monotonicity of $\sigma^*(\nu)$ will be proved once we show that $g_{\nu}(y)<0$ for all $\nu\in(0,\frac 12)$ and all $y>1$. 
This is achieved once we show that, for any fixed $\nu\in (0,\frac 12)$:
\begin{equation}\label{easy}
\lim_{y\to 1^+}g_{\nu}(y)=0
\end{equation}
and
\begin{equation}\label{lesseasy}
\dfrac{d}{dy}(y^{-2}g_{\nu}(y))<0.
\end{equation}

Now \eqref{easy} is clear from the definition of $g_{\nu}$. 
To prove \eqref{lesseasy}, we start by computing
$$
\dfrac{d}{dy}(y^{-2}g_{\nu}(y))=-\frac{2}{\nu}y^{-3-2/\nu}\left(y^4+y^{4/\nu}+y^{2/\nu}\left(\frac{2y^2}{\nu}-\frac{1-\nu}{\nu}(1+y^4)\right)\right).
$$
By completing the squares:
$$
y^{2/\nu}\left(\frac{2y^2}{\nu}-\frac{1-\nu}{\nu}(1+y^4)\right)=y^{2/\nu}\frac{1-\nu}{\nu}(y^2-1)^2+2y^{2+2/\nu}
$$
and therefore
\begin{multline*}
y^4+y^{4/\nu}+y^{2/\nu}\left(\frac{2y^2}{\nu}-\frac{1-\nu}{\nu}(1+y^4)\right)=(y^2+y^{2/\nu})^2-\frac{1-\nu}{\nu}y^{2/\nu}(y^2-1)^2\\
=\left(y^2+y^{2/\nu}-\sqrt{\frac{1-\nu}{\nu}}y^{1/\nu}(y^2-1)\right)\left(y^2+y^{2/\nu}+\sqrt{\frac{1-\nu}{\nu}}y^{1/\nu}(y^2-1)\right)
\end{multline*}
Atv this point, our claim \eqref{lesseasy} is proved if the function
$$
h_{\nu}(y):=y^2+y^{2/\nu}-\sqrt{\frac{1-\nu}{\nu}}y^{1/\nu}(y^2-1)
$$
if positive for all $\nu\in(0,\frac 12)$, $y>1$. We have that $h_{\nu}(1)=2$ for all $\nu\ne 0$. On the other hand,
$$
(y^{-2}h_{\nu}(y))'=\frac{y^{-3+1/\nu}}{\nu}k_{\nu}(y)
$$
where
$$
k_{\nu}(y):=\sqrt{\frac{1-\nu}{\nu}}(1-2\nu-y^2)+2y^{1/\nu}(1-\nu)
$$
for all $y>1$. It is then enough to show that $k_{\nu}$ is positive on $(1,\infty)$, which will be achieved by showing that it is convex and satisfies $k_{\nu}(1)>0$ and $k_{\nu}'(1)>0$. 
Now  $k_{\nu}(1)=2\sqrt{1-\nu}(\sqrt{1-\nu}-\sqrt{\nu})>0$ and $k_{\nu}'(1)=\frac{2\sqrt{1-\nu}}{\nu}(\sqrt{1-\nu}-\sqrt{\nu})>0$ for all $\nu\in(0,\frac 12)$. 
To prove convexity, observe that 
$$
k_{\nu}''(y)=2\left(\frac{(1-\nu)^2}{\nu^2}y^{-2+1/\nu}-\sqrt{\frac{1-\nu}{\nu}}\right)
$$
which is positive if and only if $y>\left(\frac{\nu}{1-\nu}\right)^{\frac{3\nu}{2(1-2\nu)}}$. But this is always true, since the function $w(\nu):=\left(\frac{\nu}{1-\nu}\right)^{\frac{3\nu}{2(1-2\nu)}}$ takes values between $0$ and $1$ on the interval $(0,\frac 12)$. In fact, $\lim_{\nu\to 0^+}w(\nu)=1$ and $w'(\nu)<0$ for all $\nu\in(0,\frac 12)$. This concludes the proof of the monotonicity of $\sigma^*(\nu)$.

\medskip

The limits of $\sigma^*(\nu)$ as $\nu\to 0^+$ and $\nu\to \frac 12^-$ follow immediately from the monotonicity of $\sigma^*(\nu)$ and from Lemma \ref{M*incr}.
\end{proof}

\begin{proof}[Proof of Theorem \ref{max_RI}] 
The proof is divided in several steps and is inspired by that of \cite[\S 3]{fs1}.

{\bf Part 1: eigenvalues of rotationally invariant annuli.} In the first part of the proof we consider a rotationally invariant annulus and compute its Steklov spectrum. In particular, we will understand how the eigenvalues depend on the conformal modulus and on the ratio of the boundary lengths.

\medskip

 We follow the approach of \cite{fs1}. Let $\Sigma$ be a rotationally invariant annulus. Then it is isometric to $[0,Z]\times\mathbb S^1$ for some $Z>0$, with metric $f(z)^2(dz^2+d\theta^2)$, $z\in[0,Z]$, $\theta\in\mathbb S^1$. If we consider, instead of $f(z)$, the flat metric
$$
f(z)=\left(1-\frac{z}{Z}\right)f(0)+\frac{z}{Z}f(Z)
$$
the  Steklov eigenvalues remain unchanged (we have performed a conformal transformation which restricts to an isometry on the boundary).  Now we proceed as usual, and separate variables, looking for solutions $u(z,\theta)$ of the form $v(z)e^ {ik\theta}$. These are written as
$$
v(z)=a\cosh(|k-\nu|z)+b\sinh(|k-\nu|z).
$$
If $\nu=0$, we are in the case of \cite{fs1}. So we will consider $\nu\in(0,\frac 12]$. The coefficients $a,b$ are found by imposing that
$$
\frac{\partial_N u}{u}|_{z=Z}-\frac{\partial_N u}{u}|_{z=0}=0
$$
which amounts to
$$
\frac{v'(Z)}{f(Z)v(Z)}+\frac{v'(0)}{f(0)v(0)}=0.
$$
Computations similar to those of Section \ref{spec:rot} allow to find that for any $k\in\mathbb Z$ we have two solutions $(a_1,b_1)$ and $(a_2,b_2)$ (up to scalar multiples), and hence we deduce that for each $k\in\mathbb Z$ we have two eigenvalues 
\begin{equation}\label{L10}
\sigma_{1k}=\frac{|k-\nu|}{2}\left(\left(\frac{1}{f(0)}+\frac{1}{f(Z)}\right)\coth(|k-\nu|Z)-\sqrt{D(Z)}\right)
\end{equation}
and
\begin{equation}\label{L20}
\sigma_{2k}=\frac{|k-\nu|}{2}\left(\left(\frac{1}{f(0)}+\frac{1}{f(Z)}\right)\coth(|k-\nu|Z)+\sqrt{D(Z)}\right),
\end{equation}
where
$$
D(Z)=\left(\frac{1}{f(0)}+\frac{1}{f(Z)}\right)^2\coth^2(|k-\nu|Z)-\frac{4}{f(0)f(Z)}.
$$
It is standard to see that $\sigma_{1k}$ and $\sigma_{2k}$ are increasing with respect to $|k-\nu|$, and moreover $\sigma_{1k}<\sigma_{2k}$ for all $k\in\mathbb Z$.

Note that these branches are all simple for $\nu\in(0,\frac 12)$, while they are all double for $\nu=\frac 12$. In fact for $\nu=\frac 12$ we have $\sigma_{ik}=\sigma_{i(1-k)}$ for all $k\in\mathbb Z$, $i=1,2$.

\medskip

Now we note that 
$$
\sigma_1(\Sigma,\nu)=\sigma_{10}
$$
for all $\nu\in(0,\frac 12)$; $\sigma_1(\Sigma,0)=0$ and $\sigma_1(\Sigma,\frac 12)=\sigma_{10}=\sigma_{11}$.

\medskip

What we will do now is to fix
$$
A=\frac{f(0)}{f(Z)},
$$
that is, to fix the ratio of the two boundary lengths. Hence the expressions for the normalized eigenvalues become
\begin{multline}\label{L11}
|\partial\Sigma|\sigma_{1k}=2\pi(f(0)+f(Z))\sigma_{1k}\\=\pi(A+1)|k-\nu|\left(\left(\frac{A+1}{A}\right)\coth(|k-\nu|Z)-\sqrt{\left(\frac{A+1}{A}\right)^2\coth^2(|k-\nu|Z)-\frac{4}{A}}\right)
\end{multline}
and
\begin{multline}\label{L22}
|\partial\Sigma|\sigma_{2k}=2\pi(f(0)+f(Z))\sigma_{2k}\\=\pi(A+1)|k-\nu|\left(\left(\frac{A+1}{A}\right)\coth(|k-\nu|Z)+\sqrt{\left(\frac{A+1}{A}\right)^2\coth^2(|k-\nu|Z)-\frac{4}{A}}\right).
\end{multline}
Standard computations allow to prove that $|\partial\Sigma|\sigma_{1k}$ is an increasing function of the variable $Z$, while  $|\partial\Sigma|\sigma_{2k}$ is a decreasing function of the variable $Z$ (see also Section \ref{spec:rot} and \cite[\S 3]{fs1}).

\medskip

For the first normalized eigenvalue, namely, $|\partial\Sigma|\sigma_{10}$ we have that, for any fixed $A>0$, it is an increasing function of $Z$. Moreover,
$$
\lim_{Z\to 0^+}|\partial\Sigma|\sigma_{10}=0
$$
and
$$
\lim_{Z\to+\infty}|\partial\Sigma|\sigma_{10}=\begin{cases}
2\pi\frac{A+1}{A}\nu\,, & A\geq 1\\
2\pi(A+1)\nu\,, & 0<A\leq 1.
\end{cases}
$$
We deduce that
$$
\sup_{Z,A}\bar\sigma_1(\Sigma,\nu)=4\pi\nu
$$
and it is achieved by taking $Z\to+\infty$ for any annulus with boundary lengths ratio equal to $1$. In the case of $\nu=\frac{1}{2}$ the same statement holds for $\bar\sigma_2(\Sigma,\nu)$. However, we have already proved the upper bound for $\bar\sigma_1(\Sigma,\nu)$ for any Riemannian annulus in Theorem \ref{max_s1}.

\medskip

{\bf Part 2: maximising $\bar\sigma_2(\Sigma,\nu)$ for fixed $A$.} Here, for a fixed value of $A$, the ratio of the boundary lengths, we prove that there exists a unique $Z$ for which $\bar\sigma_2(\Sigma,\nu)$ is maximised.

\medskip

We consider $\bar\sigma_2(\Sigma,\nu)$ when $\nu\in(0,\frac 12)$. We note that, for fixed $A$,
$$
\bar\sigma_2(\Sigma,\nu)=|\partial\Sigma|\min\left\{\sigma_{11},\sigma_{20}\right\}.
$$
Now, as we have already said, $|\partial\Sigma|\sigma_{11}$ is increasing with $Z$, and moreover we compute
$$
\lim_{Z\to 0^+}|\partial\Sigma|\sigma_{11}(\nu)=0
$$
and
$$
\lim_{Z\to +\infty}|\partial\Sigma|\sigma_{11}=\begin{cases}
2\pi\frac{A+1}{A}(1-\nu)\,, & A\geq 1\\
2\pi(A+1)(1-\nu)\,, & 0<A\leq 1.
\end{cases}
$$

On the other hand, $|\partial\Sigma|\sigma_{20}$ is decreasing with $Z$, and moreover we compute
$$
\lim_{Z\to 0^+}|\partial\Sigma|\sigma_{20}=+\infty
$$
and
$$
\lim_{Z\to +\infty}|\partial\Sigma|\sigma_{20}=\begin{cases}
2\pi\frac{A+1}{A}\nu\,, & A\geq 1\\
2\pi(A+1)\nu\,, & 0<A\leq 1.
\end{cases}
$$
We deduce that there exists a unique value of $Z=Z(A)\in(0,+\infty)$ such that
$$
\bar\sigma_2(\Sigma,\nu)=|\partial\Omega|\sigma_{11}=|\partial\Omega|\sigma_{20}.
$$
We have proved that $\bar\sigma_2(\Sigma,\nu)$ is maximised among all  rotationally invariant metrics with fixed ratio of the boundary lengths equal to $A$ for a unique value of the conformal modulus $Z/2$. In particular, in correspondence of this maximum we have that $\sigma_2(\Sigma,\nu)$ is a double eigenvalue.

\medskip

{\bf Part 3: maximising with respect to $A$.} Here we prove that among all rotationally invariant annuli that maximise $\bar\sigma_2(\Sigma,\nu)$ for a fixed value of $A$, the maximum is achieved {\it if and only if} $A=1$, i.e., for  rotationally invariant {\it symmetric} annuli. This will prove the theorem.

\medskip

Here, as in the case of $\sigma_1(\Sigma,\nu)$, we need to recognize for which $A$ we have the maximum of all these maxima, namely, understand what is
$$
\bar\sigma_2^*(\nu):=\max_{A>0}\bar\sigma_2^*(A,\nu),
$$
where $\bar\sigma_2^*(A,\nu)$ is the maximum over all $Z>0$ of $\bar\sigma_2(\Sigma,\nu)$ for a fixed $A>0$, which exists and is unique as we have proved in Part 2 of this proof. Clearly, the above definition of $\bar\sigma_2^*(\nu)$ represents the maximum of $\bar\sigma_2(\Sigma,\nu)$ over all rotationally  metrics on an annulus.

\medskip

In this part of the proof it is convenient to use a slice of the (classical, minimal) catenoid as reference annulus. Let now $a\in\mathbb R$, and consider the surface of revolution $\Sigma_a$ given by
$$
(\cosh(z-a)\cos(\theta),\cosh(z-a)\sin(\theta),z-a)
$$
for $z\in\mathbb R$, $\theta\in\mathbb S^1$. This is the catenoid centered at $(0,0,a)$.

We consider now the following two functions:
$$
u_1(z,\theta)=\cosh((1-\nu)(z-a))e^{i\theta}
$$
and
$$
u_2(z,\theta)=\sinh(\nu z).
$$
One can easily check that $\Delta_A u_i=0$ on $\Sigma_a$ for all $a$ (the functions $u_i$ play in some sense the role of $x,y,z$ in the case $\nu=0$ discussed in \cite{fs1}: they are harmonic on $\Sigma_a$ for all $a$), where $A=\nu d\theta$.

Now, let $z_1<z_2$, and let us compute $\frac{\partial_N u_i}{u_i}$ at $z=z_1$ and $z=z_2$, for $i=1,2$. It turns out that
\begin{align}
\frac{\partial_N u_1}{u_1}|_{z=z_2}&=(1-\nu)\frac{\tanh((1-\nu)(z_2-a))}{\cosh(z_2-a)},\label{c11}\\
\frac{\partial_N u_2}{u_2}|_{z=z_2}&=\nu\frac{\coth(\nu z_2)}{\cosh(z_2-a)},\label{c21}\\
\frac{\partial_N u_1}{u_1}|_{z=z_1}&=-(1-\nu)\frac{\tanh((1-\nu)(z_1-a))}{\cosh(z_1-a)},\label{c12}\\
\frac{\partial_N u_2}{u_2}|_{z=z_1}&=-\nu\frac{\coth(\nu z_1)}{\cosh(z_1-a)}.\label{c22}
\end{align}
We consider now the equation
\begin{equation}\label{t12}
F(z)=(1-\nu)\tanh((1-\nu)(z-a))-\nu\coth(\nu z)=0
\end{equation}
in the unknown $z$. We have, for any $\nu\in(0,\frac 12)$ and $a\in\mathbb R$
$$
\lim_{z\to 0^{\pm}}F(z)=\mp\infty
$$
and
$$
\lim_{z\to\pm\infty}F(z)=\pm(1-2\nu).
$$
Moreover
$$
F'(z)=\frac{\nu^2}{\sinh^2(\nu z)}+\frac{(1-\nu)^2}{\cosh^2((1-\nu)(z-a))}.
$$
We deduce that there exists exactly two values $z_1,z_2$, depending on $a,\nu$, for which $F(z_1)=F(z_2)=0$. In particular, it is not difficult  to see that, as $a\to +\infty$,
\begin{equation}\label{z1}
z_2\to+\infty,\ \ \  z_1\to-\frac{1}{\nu}\arccoth\left(\frac{1-\nu}{\nu}\right),
\end{equation}
while, as $a\to-\infty$,
\begin{equation}\label{z2}
z_2\to \frac{1}{\nu}\arccoth\left(\frac{1-\nu}{\nu}\right)\,,\ \ \ z_1\to-\infty.
\end{equation}
For $a=0$ we have $z_2=-z_1$, and $z_2$ is given as the unique positive zero of the odd function
$$
(1-\nu)\tanh((1-\nu)z)-\nu\coth(\nu z).
$$
In particular, we always have $z_1<0<z_2$.

\medskip

 We consider from now, for any fixed $a\in\mathbb R$, the slab of $\Sigma_a$ between $z_1=z_1(a,\nu)$ and $z_2=z_2(a,\nu)$ and we will denote it still by $\Sigma_a$ to simplify the notation (we recall that $z_1,z_2$ are uniquely determined by $a,\nu$).
 
  Now, let 
$$
T:=-\frac{\frac{\partial_N u_2}{u_2}|_{z=z_2}}{\frac{\partial_N u_2}{u_2}|_{z=z_1}}=-\frac{\coth(\nu z_2)\cosh(z_1-a)}{\coth(\nu z_1)\cosh(z_2-a)}.
$$
We rescale the metric in a neighborhood of $z=z_1$ by multiplying it by the factor $T$. Hence $u_1,u_2$ are Steklov eigenfunctions for this new metric (they are still $A$-harmonic) with the same eigenvalue
$$
\sigma=\frac{\partial_N u_2}{u_2}|_{z=z_2}=\nu\frac{\coth(\nu z_2)}{\cosh(z_2-a)},
$$
which is then double. Also, the length of $\partial\Sigma_a$ in this new metric is
$$
|\partial\Sigma_a|=2\pi(\cosh(z_2-a)+\cosh(z_1-a)T)
$$
In particular, $\sigma$ is the second eigenvalue: it has a second eigenfunction, $u_2$, which is real, constant on each boundary component, and of opposite sign. 

\medskip

Roughly speaking, we have identified the maximiser for $\bar\sigma_2(\Sigma,\nu)$ when the ratio of the two boundary lengths is
$$
A=\frac{\cosh(z_2-a)}{\cosh(z_1-a)T}=-\frac{\cosh^2(z_2-a)\coth(\nu z_1)}{\cosh^2(z_1-a)\coth(\nu z_2)}
$$ 
If we write $t=z-a$ in the expression of $F(z)$, we see that
$$
(1-\nu)\tanh((1-\nu)t)-\nu\coth(\nu(t+a))=0
$$
when $t=z_1-a,z_2-a$. The same analysis carried out for $z_1,z_2$ in \eqref{z1}-\eqref{z2} allows to show that, as $a\to +\infty$,
$$
z_2-a\to \frac{1}{1-\nu}\arctanh\left(\frac{\nu}{1-\nu}\right)\,,\ \ \ z_1-a\to -\infty
$$
while, as $a\to-\infty$,
$$
z_2-a\to +\infty\,,\ \ \ z_1-a\to -\frac{1}{1-\nu}\arctanh\left(\frac{\nu}{1-\nu}\right).
$$
Note that, if we define the function
$$
A(a):=-\frac{\cosh^2(z_2(a)-a)\coth(\nu z_1(a))}{\cosh^2(z_1(a)-a)\coth(\nu z_2(a))},
$$
then for any $A>0$ there exists $a\in\mathbb R$ such that $A=A(a)$. 

\medskip

In a few words, we have identified for any ratio $A$ of the boundary lengths, the height $z_2-z_1$ at which the second normalized eigenvalue is double, and hence, is maximal for that given ratio $A$.

\medskip

It remains to prove that we have the maximum if and only if $A=1$, that is, for $a=0$, i.e., when $\Sigma_a=\Sigma_0$ is symmetric. Let us write the explicit expression for $\bar\sigma_2^*(A(a),\nu)=:g(a)$ on $\Sigma_a$ (recall, we have the uniquely determined $z_1(a)<0<z_2(a)$):
\begin{multline}
g(a)=2\pi\nu(\cosh(z_2-a)+\cosh(z_1-a)T)\frac{\coth(\nu z_2)}{\cosh(z_2-a)}\\
=2\pi\nu\frac{\coth(\nu z_2)}{\coth(\nu z_1)}\cosh^2(z_1-a)\left(\frac{\coth(\nu z_1)}{\cosh^2(z_1-a)}-\frac{\coth(\nu z_2)}{\cosh^2(z_2-a)}\right).
\end{multline}
Since $z_1(a)=-z_2(-a)$, we see that it is sufficient to study $g(a)$ for $a>0$ and prove that $g(0)>g(a)$ for all $a>0$. To do so, we  prove that $g$ is {\it strictly} decreasing on $(0,+\infty)$. 

\medskip First of all, we compute $z_1'(a)$, $z_2'(a)$. Since they solve the implicit equation $F(z)=0$, we have that
\begin{equation}\label{z11}
z'(a)=-\frac{\partial_a F}{\partial_zF}|_{z=z(a)}=\frac{1}{1+\frac{\nu^2}{(1-\nu)^2}\frac{\cosh^2((z-a)(1-\nu))}{\sinh^2(z\nu)}}|_{z=z(a)}.
\end{equation}
Since $z=z(a)$ solves $F(z(a))=0$, we deduce from this identity that 
$$
\cosh^2((z(a)-a)(1-\nu))=\left(1-\frac{\nu^2}{(1-\nu)^2}\coth^2(\nu z(a))\right)^{-1}.
$$
Substituting into \eqref{z11} we get
\begin{equation}\label{z22}
z'(a)=1-\frac{\nu^2}{(1-2\nu)\sinh^2(\nu z(a))}.
\end{equation}
Hence \eqref{z22} holds with $z(a)=z_1(a)$ and $z(a)=z_2(a)$.
Now we are in position to prove that $g'(a)<0$ for all $a>0$. 

We assume that $z_i'(a)>0$. We will prove this fact at the end. This implies that $\coth(\nu z_2(a))$ is a decreasing function (recall that $z_2(a)>0$). Moreover, from \eqref{z2} we know that $(z_i(a)-a)'<0$, hence $\cosh^2(z_1(a)-a)$ is decreasing. We compute 
\begin{multline}
g'(a)
=-\nu \frac{1}{\sinh^2(\nu z_2)}\left(1-\frac{\nu^2}{(1-2\nu)\sinh^2(\nu z_2)}\right)\\
-\frac{\nu\cosh^2(z_1-a)^2\coth^(\nu z_2)}{\cosh^2(z_2-a)\cosh^2(\nu z_1)}\left(1-\frac{\nu^2}{(1-2\nu)\sinh^2(\nu z_2)}\right)\\
+2\nu\frac{\cosh^2(z_1-a)\coth(\nu z_2)\tanh(\nu z_1)}{\sinh^2(\nu z_2)\cosh^2(z_2-a)}\left(1-\frac{\nu^2}{(1-2\nu)\sinh^2(\nu z_2)}\right)\\
-\frac{2\nu^2}{1-2\nu}\frac{\cosh^2(z_1-a)\coth^2(\nu z_2)\tanh(\nu z_1)}{\cosh^2(z_2-a)}\left(\frac{\tanh(z_2-a)}{\sinh^2(\nu z_2)}-\frac{\tanh(z_1-a)}{\sinh^2(\nu z_1)}\right).
\end{multline}
The fact that the first three lines of this expression are negative follows since $a>0$, $z_2>a>0$, $z_1<0$, and since we have assumed $z_i'(a)>0$, which implies, in view of \eqref{z2}, that $1-\frac{\nu^2}{(1-2\nu)\sinh^2(\nu z_2)}>0$. It remains to prove that the last line is negative. The factor outside the parenthesis is positive. Hence we have to prove that the function
$$
h(a)=\frac{\tanh(z(a)-a)}{\sinh^2(\nu z(a))}
$$
is decreasing. We recall that $|z_i(a)|>c(\nu)>0$ for all $a>0$. Hence the function is smooth. We compute
\begin{multline}
h'(a)=\frac{\nu}{\sinh^2(\nu z(a))}\left(-2\coth(\nu z(a))\tanh(z(a)-a)\left(1-\frac{\nu^2}{(1-2\nu)\sinh^2(\nu z(a))}\right)\right.\\
\left.-\frac{\nu}{(1-2\nu)\cosh^2(z(a)-a)\sinh^2(\nu z(a))}\right).
\end{multline}
Again, we see that this quantity is negative since $z_i'(a)>0$ and that $\coth(\nu z_i(a))\tanh(z_i(a)-a)>0$ for all $a>0$.

\medskip

The proof is concluded provided $z_i'(a)>0$. Now, we have from \eqref{z2} that
$$
z'(a)=1-\frac{\nu^2}{(1-2\nu)\sinh^2(\nu z(a))}.
$$
In particular,
$$
z'(0)=1-\frac{\nu^2}{(1-2\nu)\sinh^2(\nu z(0))}
$$
and $z(0)$ solves $(1-\nu)\tanh((1-\nu)z(0))=\nu\coth(\nu z(0))$. We find, using this relation, that
$$
\sinh^2(\nu z(0))=\frac{1}{\left(\frac{1-\nu}{\nu}\right)^2\tanh^2((1-\nu)z(0))-1}>\frac{\nu^2}{1-2\nu}.
$$
Hence $z'(0)>0$. Computing $\frac{d}{da}\sinh^2(\nu z(a))$ we find that its derivative is
$$
\nu\cosh(\nu z(a))z'(a),
$$
so that $\sinh^2(\nu z(a))$ is increasing near $a=0$. Hence we will always have $z'(a)>0$.
\end{proof}

\section{Proofs of Section \ref{geom1}}\label{app:6}

In this section we prove Theorem \ref{geom-asurf}, Proposition \ref{Madecr} and Theorem \ref{geom_car_asurf}.

\begin{proof}[Proof of Theorem \ref{geom-asurf}]
We first prove the equivalence of $a)$ and $b)$. A unit normal vector field to $\Sigma$ is given by
$$
N(t,\theta)=\dfrac{1}{\sqrt{1+\rho'(t)^2}}(-\cos(\theta),-\sin(\theta), \rho'(t)).
$$
A standard computation with the second fundamental form shows that the principal curvatures are given by
\begin{equation}\label{pc}
\kappa_1(t,\theta)=-\dfrac{\rho''(t)}{(1+\rho'(t)^2)^{3/2}}, \quad \kappa_2(t,\theta)=\dfrac{1}{\rho(t)\sqrt{1+\rho'(t)^2}}.
\end{equation}
Note that $\kappa_1$ is the principal curvature in the axial direction $\derive{}{t}$ (it equals, in fact, the curvature of the profile at $t$), while $\kappa_2$ is the principal curvature in the rotational direction $\derive{}{\theta}$. One sees immediately, from the defining ODE \eqref{asurf-ODE},  that $\kappa_1+\alpha \kappa_2=0$.

\medskip

We prove now the equivalence of $a)$ and $c)$. It is well-known (see e.g., \cite{sir}) that a surface is critical for the weighted area 
$
\int_{\Sigma}w\,dv,
$
where $w= e^{-f}$ is a smooth positive density, if and only if the weighted mean curvature
$$
H_{f}:= H+\derive fN
$$
vanishes identically. Here $N$ is the unit normal vector and $H$ is the trace if the second fundamental form of $\Sigma$ with respect to $N$ (that is, $H=\kappa_1+\kappa_2$). Taking
$w=\rho^{\alpha-1}$ amounts to taking $f=(1-\alpha)\log\rho$, so that:
\begin{equation}\label{weightedmc}
H_{f}=H+\dfrac{1-\alpha}{\rho}\scal{\bar\nabla\rho}{N},
\end{equation}
where $\bar\nabla$ denotes the usual gradient in $\mathbb R^3$. We observe that the function $\scal{\bar\nabla\rho}{N}$ does not depend on $\theta$, and then we can compute it for $\theta=0$, where we have:
$$
\bar\nabla\rho(t,0)=(1,0,0), \quad N(t,0)=\dfrac{1}{\sqrt{1+\rho'(t)^2}}(-1,0, \rho'(t)),
$$
hence
$$
\scal{\bar\nabla\rho}{N}=-\dfrac{1}{\sqrt{1+\rho'(t)^2}}.
$$
Recalling that $H=\kappa_1+\kappa_2$ we get from \eqref{pc} and \eqref{weightedmc}:
$$
H_{f}=\dfrac{1}{\rho(1+\rho'^2)^{3/2}}\Big(-\rho\rho''+\alpha(1+\rho'^2)\Big)
$$
which vanishes if and only if $a)$ holds. 

\smallskip

We prove now the equivalence of $c)$ and $d)$. Here we state some preliminary facts.  Consider the potential one-form in $\mathbb R^3$
$$
A=\dfrac{\alpha}{\rho^2}(-xdy+ydx),
$$
(recall that $\rho$ is the distance from the $z$-axis, namely $\rho=\sqrt{x^2+y^2}$).

 We write $\bar\Delta_Au$ for the magnetic Laplacian in $\mathbb R^3$ and let $u=\rho^{\alpha}$. Also, we write $\bar\Delta$ for the usual Laplacian in $\mathbb R^3$. Then, since $A$ is co-closed:
$$
\bar\Delta_Au=\bar\Delta u+\abs{A}^2u+2i\scal{du}{A}.
$$
Now $du=\alpha\rho^{\alpha-1}d\rho$ and since $A(\bar\nabla\rho)=0$ we see that $\scal{du}{A}=0$. From the formula
$$
\bar\Delta(\psi\circ\rho)=-(\psi''\circ\rho)\abs{\bar\nabla\rho}^2+(\psi'\circ\rho)\bar\Delta\rho,
$$
we obtain, since $\bar\Delta\rho=-\frac 1{\rho}$, the formula
$$
\bar\Delta(\psi\circ\rho)=-(\psi''\circ\rho)-\frac{1}{\rho}(\psi'\circ\rho).
$$
Since $\abs{A}^2=\frac{\alpha^2}{\rho^2}$ we arrive, writing $\psi=\psi(\rho)$, at 
$$
\bar\Delta_Au=(-\psi''-\frac{\psi'}{\rho}+\frac{\alpha^2}{\rho^2}\psi)\circ\rho.
$$
Note that, when applied to $\psi=\rho^{\alpha}$ we get $\bar\Delta_Au=0$.

\smallskip

Next, we observe that $\bar\nabla A(N,N)=0$ everywhere, because $A$ is a function times a Killing form. 
Since $\Sigma$ is rotationally invariant we also have $A(N)=0$ everywhere on $\Sigma$ and $\bar\nabla^2\rho(N,N)=0$, where $\bar\nabla^2$ denotes the usual Hessian in $\mathbb R^3$.  Since $A(N)=0$, we observe that viewing $A$ as a vector field in $\mathbb R^3$, it is everywhere tangent to any rotational surface and that its projection onto $\Sigma$ is itself. As usual, we denote by $\Delta_A$ the magnetic Laplacian of the surface $\Sigma$ endowed with (the restriction of) $A$ as well, which is closed with flux $\alpha$. We want to show that $\Delta_Au=0$. 

\medskip

For any function $u$ on $\mathbb R^3$ the following decomposition of $\bar\Delta_A$ holds (it follows from the analogous, well-known decomposition of $\bar\Delta$):
\begin{equation}\label{blaplacian}
\bar\Delta_Au=\Delta_Au-\bar\nabla^2u(N,N)+H\scal{\bar\nabla u}{N}.
\end{equation}
If $u=\psi\circ\rho$ we see that $\bar\nabla^2u=(\psi''\circ\rho)d\rho\otimes d\rho+(\psi'\circ\rho)\bar\nabla^2\rho$. Hence:
$$
\bar\nabla^2u(N,N)=\alpha(\alpha-1)\rho^{\alpha-2}\scal{\bar\nabla\rho}{N}^2.
$$
One arrives at the formula 
$$
-\bar\nabla^2u(N,N)+H\scal{\bar\nabla u}{N}=-\alpha\rho^{\alpha -1}\scal{\bar\nabla\rho}{N}H_f
$$ 
where $H_f=H-\dfrac{\alpha-1}{\rho}\scal{\bar\nabla\rho}{N}$ is the weighted mean curvature
computed in \eqref{weightedmc}. From \eqref{blaplacian} we get
\begin{equation}\label{blaplaciantwo}
\bar\Delta_Au=\Delta_Au-\alpha\rho^{\alpha -1}\scal{\bar\nabla\rho}{N}H_f.
\end{equation}
Since $\bar\Delta_Au=0$, we conclude that $\Delta_Au=0$ if and only if $H_f=0$ as asserted. 
\end{proof}

In order to prove Proposition \ref{Madecr}, we need a few preliminary definitions and results.

\medskip

Let $\alpha>0$. The unique $\alpha$-surface satisfying the conditions:
\begin{equation}\label{reference}
\begin{cases}
\rho\rho''=\alpha(1+\rho'^2) &  {\rm in\ } (-\infty,\infty)\\
\rho(0)=1\,, \rho'(0)=0
\end{cases}
\end{equation}
will be called the {\it reference $\alpha$-surface}, and denoted $\Sigma(\alpha)$.
It is the unique $\alpha$-surface with equatorial radius equal to $1$. For the reader's convenience, we recall that $\Sigma(\alpha)$ is parametrized as $(\rho(t)\cos(\theta),\rho(t)\sin(\theta),t)$, $(t,\theta)\in (-\infty,\infty)\times\mathbb S^1$ and $\rho(t)$ (the distance from the $z$-axis) is the solution of \eqref{reference}.

\begin{lemma} Let $\Sigma(\alpha)$ be the reference $\alpha$-surface as in \eqref{reference}. Then:
\begin{enumerate}[a)]
\item $\rho'(t)=\sqrt{\rho^{2\alpha}(t)-1}$ for all $t$. In particular $\rho''=\alpha\rho^{2\alpha-1}$.
\item Let $\Sigma_T(\alpha)$ be the slice where $-T\leq t\leq T$. Then the function $u=\rho^{\alpha}$ defined in $\mathbb R^3$ restricts to a magnetic Steklov eigenfunction on $\Sigma_T(\alpha)$ with flux $\alpha$ associated with the eigenvalue
$$
\sigma(\Sigma_T(\alpha),\alpha)=\dfrac{\alpha}{\rho(T)}\sqrt{1-\dfrac{1}{\rho(T)^{2\alpha}}}.
$$
As $|\partial\Sigma_T(\alpha)|=4\pi\rho(T)$, we see that
$$
\bar\sigma(\Sigma_T(\alpha),\alpha):=|\partial\Sigma_T(\alpha)|\sigma(\Sigma_T(\alpha),\alpha)=4\pi\alpha\sqrt{1-\dfrac{1}{\rho(T)^{2\alpha}}}
$$
is the normalized eigenvalue associated to the eigenfunction $u=\rho^{\alpha}$.
\end{enumerate}
\end{lemma}
\begin{proof} We prove $a)$. Consider the function $\psi(t)=\rho(t)\sqrt{1+\rho'(t)^2}$. A calculation shows that
$
\frac{\psi'}{\psi}=(1+\alpha)\frac{\rho'}{\rho}.
$
Integrating this identity on $[0,t]$ we get the assertion. 

\medskip

We prove now $b)$. From Theorem \ref{geom-asurf} we know that $\rho^{\alpha}$ is $\Delta_A$-harmonic on $\Sigma(\alpha)$. The assertion now follows by a straightforward calculation of the normal derivative of $\rho^{\alpha}$ at the two components of the boundary (see also the proof of Theorem \ref{geom-asurf}) .
\end{proof}

Now there is a unique, finite value of $T$ for which the slice $\Sigma_T(\alpha)$ is free boundary in some ball in $\mathbb R^3$. This can be immediately deduced by the fact that the profile curve is convex and moreover $(\rho^{-1})':(1,+\infty)\to(0,+\infty)$ decreases from $+\infty$ to $0$ (see \cite{collharr}, see also the proof of Lemma \ref{l40} here below). We denote this value $T(\alpha)$, and let $R(\alpha):= \rho(T_{\alpha})$.  We call $T(\alpha)$ the {\it height} and $R(\alpha)$ the {\it radius} of this free boundary slice. In the next lemma we give some information on the behavior of $R(\alpha)$ as function of $\alpha$.

\begin{lemma}\label{l40} Let $R(\alpha)$ be the radius of the slice of the reference $\alpha$-surface $\Sigma(\alpha)$ which is free-boundary in some ball. Then:
\begin{enumerate}[a)]
\item $R(\alpha)$ is decreasing on $(0,\infty)$.
\item  $\lim_{\alpha\to+\infty}R(\alpha)=1$.
\end{enumerate}
\end{lemma}

\begin{proof} 

Through this proof we denote by $\rho_{\alpha}$ the profile function of $\Sigma(\alpha)$ given by \eqref{reference} (i.e., we highlight the dependence on $\alpha$).

\medskip

We start by proving $a)$.  We let $\phi_{\alpha}:=\rho_{\alpha}^{-1}:[1,\infty]\to [0,\infty)$; we have the explicit expression
$$
\phi_{\alpha}(r)=\int_1^r\dfrac{1}{\sqrt{x^{2\alpha}-1}}dx, \quad \phi'_{\alpha}(r)=\dfrac{1}{\sqrt{r^{2\alpha}-1}}.
$$
This follows from the differential equation \eqref{reference} (see also \cite{collharr}). 
Since the slice $\Sigma_T(\alpha)$ is free boundary in some ball, we have that $R(\alpha)$ is the zero of the function $F_{\alpha}(r)=r\phi'_{\alpha}(r)-\phi_{\alpha}(r)$. If we let
$$
G(r,\alpha)=\dfrac{r}{\sqrt{r^{2\alpha}-1}}-\phi_{\alpha}(r)
$$
we see that $G(R(\alpha),\alpha))=0$ for all $\alpha$.  
We can apply the Implicit Function Theorem and compute
$$
R'(\alpha)=-\frac{r}{\alpha}\log(r)+\frac{(r^{2\alpha}-1)^{3/2}}{\alpha}\int_1^r\frac{t^{2\alpha}\log(t)}{r^{2\alpha}(t^{2\alpha}-1)^{3/2}}dt|_{r=R(\alpha)}
$$
Since we know that $R(\alpha)=\sqrt{R(\alpha)^{2\alpha}-1}\int_1^{R(\alpha)}\frac{1}{\sqrt{t^{2\alpha}-1}}dt$, we have
\begin{multline}
R'(\alpha)=-\frac{\sqrt{r^{2\alpha}-1}}{\alpha}\log(r)\int_1^r\frac{1}{\sqrt{t^{2\alpha}-1}}dt+\frac{(r^{2\alpha}-1)^{3/2}}{\alpha}\int_1^r\frac{t^{2\alpha}\log(t)}{r^{2\alpha}(t^{2\alpha}-1)^{3/2}}dt|_{r=R(\alpha)}\\
=-\frac{\sqrt{r^{2\alpha}-1}}{\alpha}\int_1^r\frac{\log(r)}{\sqrt{t^{2\alpha}-1}}\left(1-\frac{h(t)}{h(r)}\right)dt|_{r=R(\alpha)}
\end{multline}
where
$$
h(t)=\frac{t^{2\alpha}\log(t)}{t^{2\alpha}-1}.
$$
Now, $h'(t)=\frac{t^{2\alpha}(t^{2\alpha}-1-2\alpha\log(t))}{t(t^{2\alpha}-1)^2}$. Considering $h_1(t):=t^{2\alpha}-1-2\alpha\log(t)$ we see immediately that $h_1(1)=0$ and $h_1'(t)=2\alpha t^{-1}(t^{2\alpha}-1)>0$ for all $t>1$. Hence $h(t)$ is increasing and therefore $h(t)<h(r)$ for all $1<t<r$. Hence $1-\frac{h(t)}{h(r)}>0$ and the claim is proved.

\medskip

We prove now $b)$. By definition, if $T(\alpha)$ is the height of the reference free boundary $\alpha$-surface, we have $\phi_{\alpha}(R(\alpha))=T(\alpha)$ hence
$$
T(\alpha)=\int_1^{R(\alpha)}\dfrac{1}{\sqrt{x^{2\alpha}-1}}\,dx.
$$
Since $R(\alpha)$ is decreasing and positive, we have $\lim_{\alpha\to\infty}R(\alpha)=c$ for some $c$.  Note that $c\geq 1$. We prove the assertion by showing that the assumption $c>1$ leads to a contradiction. Note that, in that case, $R(\alpha)\geq c>1$ for all $\alpha$ and
$$
T(\alpha)\geq \int_1^{c}\dfrac{1}{\sqrt{x^{2\alpha}-1}}\,dx\geq C>0
$$
for all $\alpha$. Then, it is enough to prove that 
$$
\lim_{\alpha\to\infty}T(\alpha)=0.
$$
The free boundary condition implies that $T(\alpha)$ is the unique zero of the function
$F_{\alpha}(t)=t\rho_{\alpha}'(t)-\rho_{\alpha}(t)$; since $\rho_{\alpha}''=\alpha\rho_{\alpha}^{2\alpha-1}$ we see $F'_{\alpha}(t)=\alpha t\rho_{\alpha}^{2\alpha-1}(t)$. Integrating on $[0,T(\alpha)]$ we obtain the identity:
$$
\dfrac 1{\alpha}=\int_0^{T(\alpha)}t\rho_{\alpha}^{2\alpha-1}(t)\,dt.
$$
When $\alpha\geq\frac12$, since $\rho_{\alpha}\geq 1$, we obtain $\rho_{\alpha}^{2\alpha-1}(t)\geq 1$, hence 
$$
T(\alpha)\leq\sqrt{\dfrac{2}{\alpha}},
$$
which tends to zero when $\alpha\to\infty$. The proof is complete.
\end{proof} 

\color{black}
Recall that the critical $\alpha$-surface is the free-boundary $\alpha$-surface meeting the {\it unit} ball orthogonally; it is obtained by suitably scaling $\Sigma(\alpha)$. We denote it by $\Sigma_c(\alpha)$, and we let $T_c(\alpha)$ be its height and $R_c(\alpha)$ its radius. 

\begin{lemma}
\begin{enumerate}[a)]
\item One has the formula
$
R_c(\alpha)=\sqrt{1-\dfrac{1}{R(\alpha)^{2\alpha}}}.
$
\item Let $M(\alpha)$ be the conformal modulus of the critical $\alpha$-surface $\Sigma_c(\alpha)$. Then:
$$
\tanh(\alpha M(\alpha))=\sqrt{1-\dfrac{1}{R(\alpha)^{2\alpha}}}.
$$
\end{enumerate}
\end{lemma}

\begin{proof}
We prove first $a)$. Recall that normalized Steklov eigenvalues are invariant by homotheties; hence for the eigenvalues corresponding to the eigenfunction $u=\rho^{\alpha}$ (which is an eigenfunction for both $\Sigma_T(\alpha)$ and $\Sigma_c(\alpha)$) one sees that:
$$
|\partial\Sigma_T(\alpha)|\sigma(\Sigma_T(\alpha),\alpha)=
|\partial\Sigma_c(\alpha)|\sigma(\Sigma_c(\alpha),\alpha).
$$
By the previous lemma, the left hand side is equal to 
$
4\pi\alpha\sqrt{1-\dfrac{1}{R(\alpha)^{2\alpha}}};
$
on the other hand we know from Theorem \ref{geom_car_asurf} that $\sigma(\Sigma_c(\alpha),\alpha)=\alpha$ and 
$\abs{\bd\Sigma_c(\alpha)}=4\pi R_c(\alpha)$. Equating we get the assertion. 

\medskip

We prove now $b)$. Since $\Sigma_c(\alpha)$ is a rotationally invariant, symmetric annulus,  it is $\sigma$-homothetic to the flat cylinder $\Gamma_{M(\alpha)}$ and then the respective normalized spectra are the same. 
For any $\alpha>0$ both $4\pi\alpha R_c(\alpha)$ and $4\pi\alpha\tanh(\alpha M(\alpha))$ belong to the normalized spectrum of $\Sigma_c(\alpha)$. They actually coincide, since $4\pi\tanh(\alpha M(\alpha))$ is the only eigenvalue of $\Gamma_{M(\alpha)}$ with associated eigenfunction which is real, symmetric (i.e., the its value is the same at the two boundary components) and constant on each boundary component. Hence it must coincide with the pull-back of $\rho^{\alpha}$ which is the eigenfunction corresponding to the eigenvalue $4\alpha\pi R_c(\alpha)$ on $\Sigma_c(\alpha)$. This shows that
$$
\tanh(\alpha M(\alpha))=R_c(\alpha),
$$
and the assertion follows from part $a)$. 
\end{proof}

We are now ready to prove Proposition \ref{Madecr}.

\begin{proof}[Proof of Proposition \ref{Madecr}]

Differentiating the identity
$
M(\alpha)=\dfrac{1}{\alpha}{\rm arctanh}\Big(\sqrt{1-\dfrac{1}{R(\alpha)^{2\alpha}}}\Big)
$
 we get
\begin{multline}
M'(\alpha)=\dfrac{1}{\alpha^2R(\alpha)\sqrt{1-R(\alpha)^{-2\alpha}}}\cdot\\
\cdot\Big[R(\alpha)\Big(\log(R(\alpha)^{2\alpha})-\sqrt{1-R(\alpha)^{-2\alpha}}\cdot {\rm arctanh}\Big(\sqrt{1-R(\alpha)^{-2\alpha}}\Big)\Big)+\alpha^2R'(\alpha)\Big]
\end{multline}
Since $R'(\alpha)\leq 0$ it is enough to show that 
$$
h(x):= \log x-\sqrt{1-x^{-2}}{\rm arctanh}(\sqrt{1-x^{-2}})<0,
$$
where we have put $x=R(\alpha)^{\alpha}>1$. Now
$
\lim_{x\to 1}h(x)=0,
$
and
$$
h'(x)=-\dfrac{{\rm arccoth}(\dfrac{x}{\sqrt{x^2-1}})}{x^2\sqrt{x^2-1}}<0
$$
which implies that $h(x)<0$ for all $x$.

\medskip Now we compute the limits of $M(\alpha)$. We start by showing that $\lim_{\alpha\to+\infty}M(\alpha)=0$. Recall that
$
\tanh(\alpha M(\alpha))=\sqrt{1-\dfrac{1}{R(\alpha)^{2\alpha}}}
$
hence
$$
\dfrac{1}{R(\alpha)^{2\alpha}}=1-\tanh^2(\alpha M(\alpha))=\dfrac{1}{\cosh^2(\alpha M(\alpha))}
$$
and therefore
$$
R(\alpha)^{\alpha}=\cosh(\alpha M(\alpha)).
$$
Assume that $\lim_{\alpha \to\infty} M(\alpha)=C>0$. Then, since $M(\alpha)$ is decreasing, we would have $M(\alpha)\geq C$ for all $\alpha$. This implies that
$$
R(\alpha)^{\alpha}\geq \cosh(\alpha C)\geq \dfrac{e^{\alpha C}}{2},
$$
which in turn implies that
$
R(\alpha)\geq \dfrac{e^{C}}{2^{\frac 1{\alpha}}}.
$
Passing to the limit as $\alpha\to\infty$ we get:
$$
\lim_{\alpha\to\infty}R(\alpha)\geq e^{C}>1
$$
contradicting the fact that $R(\alpha)\to 1$ as $\alpha\to\infty$.

\medskip

Finally we prove that $\lim_{\alpha\to 0^+}M(\alpha)=+\infty$. We use again the identity
$
R(\alpha)^{\alpha}=\cosh(\alpha M(\alpha))
$
and assume that $M(\alpha)\to C<+\infty$ as $\alpha\to 0$. Then, by monotonicity, we have $M(\alpha)\leq C$ for all $\alpha$ and 
$$
R(\alpha)\leq (\cosh(\alpha C))^{\frac 1{\alpha}}=e^{\frac 1{\alpha}\log(\cosh(C\alpha))}.
$$
Near $\alpha=0$ we see 
$
\log(\cosh(C\alpha))=\dfrac{C^2}{2}\alpha^2+O(\alpha^3)
$
which implies that
$
\lim_{\alpha\to 0}R(\alpha)\leq 1.
$
This is evidently impossible, because then, since $R(\alpha)$ is decreasing and 
tends to $1$ as $\alpha\to \infty$, it would give $R(\alpha)=1$ for all $\alpha$.
\end{proof}

We conclude this section with the proof of  Theorem \ref{geom_car_asurf}.

\begin{proof}[Proof of Theorem \ref{geom_car_asurf}]
 Assume that $\Sigma$ is a critical $\alpha$-surface. Then $\rho^{\alpha}$ is $A$-harmonic on $\Sigma$; let $N$ denote the outer unit normal to $\partial\Sigma$. Then
$$
d^A\rho^{\alpha}(N)=d\rho^{\alpha}(N)-i\rho^{\alpha}A(N).
$$
Now $A(N)=0$ because, since $\Sigma$ is rotationally invariant, one has that $A$ is tangential to the boundary of $\Sigma$, hence
$$
d^A\rho^{\alpha}(N)=d\rho^{\alpha}(N)=\alpha\rho^{\alpha-1}\scal{\bar\nabla\rho}{N}.
$$
Now one has $\rho=\sqrt{x^2+y^2}$ hence 
$$
\bar\nabla\rho=\dfrac{1}{\rho}\left(x\frac{\partial}{\partial x}+y\frac{\partial}{\partial y}\right).
$$
On the other hand $\bd\Sigma\perp\bd B$ so that
$$
N=X=x\derive{}{x}+y\derive{}{y}+z\derive{}{z}
$$
and therefore
$$
\scal{\bar\nabla\rho}{N}=\rho.
$$
The conclusion is that
$$
d^A\rho^{\alpha}(N)=\alpha\rho^{\alpha},
$$
which proves the first half of the theorem. 

\medskip

Conversely, assume that $\rho^{\alpha}$ is a Steklov eigenfunction associated to $\alpha$. Since $\rho^{\alpha}$ is $\bar\Delta_A$-harmonic in $\mathbb R^3$ and restricts to a $\Delta_A$-harmonic function  on $\Sigma$ (recall $A=\frac{\alpha}{x^2+y^2}(-ydx+xdy)$), we see that
 $\Sigma$   must be an $\alpha$-surface by Theorem \ref{geom-asurf}. It remains to show that $\partial\Sigma\perp\bd B$ or, equivalently, that $N=X$ at every point $p$ of the boundary, where $X$ is the position vector. Now, by assumption, $d^A\rho^{\alpha}(N)=\alpha\rho^{\alpha}$ hence one must have
$
\scal{\bar\nabla\rho}{N}=\rho.
$
On the other hand we also have $\scal{\bar\nabla\rho}{X}=\rho$ hence 
$$
\scal{\bar\nabla\rho}{X}=\scal{\bar\nabla\rho}{N}.
$$ 
The vectors $N,X$ and $\bar\nabla\rho$ are unit vectors and, applied at $p\in\bd\Sigma$, are contained in the same plane $\Pi$, which we can assume that it is the $xz$-plane (the plane containing the profile); moreover $X$ and $N$ form an acute angle and $N$ is tangent to the profile at $p$.  Assume $X\ne N$, so that $(X,N)$ form a basis of $\Pi$ and one has $\bar\nabla\rho=aX+bN$ for some $a,b\in\mathbb R$. Then we have
$$
0=\scal{\bar\nabla\rho}{X-N}=(a-b)(1-\scal XN).
$$
Now it is clear that $a\ne b$, otherwise $\bar\nabla\rho=\derive{}{x}$ would belong to the cone determined by $X$ and $N$, which is clearly impossible because $X$ and $N$ form a positive, acute angle with $\derive{}{ x}$. Then $\scal XN=1$ and since $X$ and $N$ are unit vectors we see that $X=N$. 
\end{proof}

\section{Proofs of Section \ref{geom2}}\label{app:7}

{\bf Proof of Theorem \ref{FSs2}.} a) We compute the matrix of the induced metric in the coordinates $(t,\theta)$. A basis of the tangent space at any point of $\Sigma$ is then $(\bd_t,\bd_{\theta})$
where  $\bd_t=\derive{}{t}, \bd_\theta=\derive{}{\theta}$.  We have
$$
g_{\nu,11}=g_{\nu}(\bd_t,\bd_t)=\Re\left(h\left(d^AF(\bd_t), d^AF(\bd_t)\right)\right)
$$
and similarly for the other components. A straightforward calculation using the identities $A(\bd_t)=0, A(\bd_{\theta})=\nu$ gives:
$$
\twosystem
{d^AF(\bd_t)=\Big(a\sinh((1-\nu)t)e^{i\theta}, a\cosh(\nu t)\Big)}
{d^AF(\bd_{\theta})=\Big(i(1-\nu)u_1,-i\nu u_2\Big)}
$$
and one arrives at the expressions:
$$
g_{\nu,11}=a^2\Big(\sinh^2((1-\nu)t)+\cosh^2(\nu t)\Big), \quad g_{\nu,22}=a^2\Big(\cosh^2((1-\nu)t)+\sinh^2(\nu t)\Big), \quad g_{\nu,12}=0,
$$
It is immediate to observe that $g_{\nu,11}=g_{\nu,22}$ hence the metric is conformal to the flat metric. 

\smallskip

b)  We see that a vector normal to $F(\Gamma_{M^*})$ inside $\mathbb C\times \mathbb R$ is given by
$$
U(t,\theta)=(-\cosh(\nu t)e^{i\theta},\sinh((1-\nu)t)),
$$
being orthogonal to both $dF(\bd_t)$ and $dF(\bd_{\theta})$.
A computation shows that
$$
\scal{F}{U}=\dfrac{a^2}{\nu(1-\nu)}\Big(-\nu\cosh((1-\nu)t)\cosh(\nu t)+(1-\nu)\sinh (\nu t)\sinh((1-\nu)t)\Big)
$$
which is zero when $t=\pm M^*$ because $M^*$ solves $(1-\nu)\tanh((1-\nu)M^*)=\nu\coth(\nu M^*)$. Hence the immersion is free boundary.

\smallskip

c) is already proved in Corollary \ref{cor_max_RI}.

\medskip

{\bf Proof of Theorem  \ref{FSmagnetic}.} For the proof, we consider cartesian coordinates $x,y,z$ on the unit ball $B$ of $\mathbb R^3$, identified with $\mathbb C\times\mathbb R$, and write $u_1=x+iy$ and $u_2=z$. Note that $x,y,z$ are real valued functions on $\Sigma$. Let $N_{\bd\Sigma}$ be the exterior unit normal vector to $\bd\Sigma$ in $\Sigma$: it is a real vector tangent to $\Sigma$; as the boundary of $\Sigma$ is a parallel circle, it is an integral curve of the dual vector field of $A$, and we have immediately that $A(N_{\bd\Sigma})=0$. This means that,  for any complex function $u$ on $\Sigma$ we have 
$$
d^Au(N_{\bd\Sigma})=du(N_{\bd\Sigma}).
$$
Let $X=x\bd_x+y\bd_y+z\bd_z$ be the position vector. Note that the boundary of $\Sigma$ meets the boundary of $B$ orthogonally if and only if $N_{\bd\Sigma}=X$ at every point of $\bd\Sigma$.

Assume first that $\Sigma$ is $A$-minimal and meets the boundary orthogonally. Then we have $\Delta_Au_1=\Delta_Au_2=0$ on $\Sigma$; on $\bd\Sigma$ we have: 
$$
d^Au_1(N_{\Sigma})=du_1(N_{\Sigma})=du_1(X)=(dx+idy)(x\bd_x+y\bd_y+z\bd_z)=x+iy=u_1
$$
which proves that $u_1$ is a Steklov eigenfunction with eigenvalue $1$. The same proof applies to $u_2$, which proves the first part. 

Now assume that $u_1$ and $u_2$ are Steklov eigenfunctions with eigenvalue $1$. This means that 
$\Delta_Au_1=\Delta_Au_2=0$ on $\Sigma$, hence $\Sigma$ is $A$-minimal; it remains to show that $N_{\bd\Sigma}=X$. For that we use the assumption that $d^Au_k(N_{\bd\Sigma})=u_k$ on $\bd\Sigma$ for $k=1,2$. Now:
$$
x+iy=u_1=d^Au_1(N_{\bd\Sigma})=du_1(N_{\bd\Sigma})=(dx+idy)(N_{\bd\Sigma})
$$
which shows that $dx(N_{\bd\Sigma})=x, dy(N_{\bd\Sigma})=y$; in the same way, one proves that $dz(N_{\bd\Sigma})=z$. Finally, since $dx(N_{\bd\Sigma})=\scal{N_{\bd\Sigma}}{\bd_x}$ etc. we see:
$$
\begin{aligned}
N_{\bd\Sigma}&=dx(N_{\bd\Sigma})\bd_x+dy(N_{\bd\Sigma})\bd_y+dz(N_{\bd\Sigma})\bd_z\\
&=x\bd_x+y\bd_y+z\bd_z\\
&=X
\end{aligned}
$$
and the proof is complete.

\section*{Acknowledgements} The authors acknowledge support from the INDAM-GNSAGA project ``Analisi Geometrica: Equazioni alle Derivate Parziali e Teoria delle Sottovarietà''. The first author acknowledges support from the project ``Perturbation problems and asymptotics for elliptic differential equations: variational and potential theoretic methods'' funded by the MUR Progetti di Ricerca di Rilevante Interesse Nazionale (PRIN) Bando 2022 grant 2022SENJZ3. The second author acknowledges support from the project ``Real and Complex Manifolds: Geometry and Holomorphic Dynamics'' funded by the MUR Progetti di Ricerca di Rilevante Interesse Nazionale (PRIN) Bando 2022 grant 2022AP8HZ9. 

\bibliographystyle{plain}
\bibliography{biblioCPSmagnetic1}

\def\cprime{$'$} \def\cprime{$'$} \def\cprime{$'$} \def\cprime{$'$}
  \def\cprime{$'$}
\begin{thebibliography}{10}

\bibitem{AhBo}
Y.~Aharonov and D.~Bohm.
\newblock Significance of electromagnetic potentials in the quantum theory.
\newblock {\em Phys. Rev.}, 115:485--491, Aug 1959.

\bibitem{bucur}
Dorin Bucur and Micka\"{e}l Nahon.
\newblock Stability and instability issues of the {W}einstock inequality.
\newblock {\em Trans. Amer. Math. Soc.}, 374(3):2201--2223, 2021.

\bibitem{cps}
Bruno Colbois, Luigi Provenzano, and Alessandro Savo.
\newblock Isoperimetric inequalities for the magnetic {N}eumann and {S}teklov
  problems with {A}haronov-{B}ohm magnetic potential.
\newblock {\em J. Geom. Anal.}, 32(11):Paper No. 285, 38, 2022.

\bibitem{collharr}
Vincent Coll and Michael Harrison.
\newblock Hypersurfaces of revolution with proportional principal curvatures.
\newblock {\em Adv. Geom.}, 13(3):485--496, 2013.

\bibitem{ei1}
A.~El~Soufi and S.~Ilias.
\newblock Immersions minimales, premi\`{e}re valeur propre du laplacien et
  volume conforme.
\newblock {\em Math. Ann.}, 275(2):257--267, 1986.

\bibitem{ei2}
Ahmad El~Soufi and Sa\"{i}d Ilias.
\newblock Riemannian manifolds admitting isometric immersions by their first
  eigenfunctions.
\newblock {\em Pacific J. Math.}, 195(1):91--99, 2000.

\bibitem{FL_Steklov}
Alberto Ferrero and Pier~Domenico Lamberti.
\newblock Spectral stability of the {S}teklov problem.
\newblock {\em Nonlinear Anal.}, 222:Paper No. 112989, 33, 2022.

\bibitem{fs1}
Ailana Fraser and Richard Schoen.
\newblock The first {S}teklov eigenvalue, conformal geometry, and minimal
  surfaces.
\newblock {\em Adv. Math.}, 226(5):4011--4030, 2011.

\bibitem{fs2}
Ailana Fraser and Richard Schoen.
\newblock Sharp eigenvalue bounds and minimal surfaces in the ball.
\newblock {\em Invent. Math.}, 203(3):823--890, 2016.

\bibitem{gipo}
Alexandre Girouard and Iosif Polterovich.
\newblock Upper bounds for {S}teklov eigenvalues on surfaces.
\newblock {\em Electron. Res. Announc. Math. Sci.}, 19:77--85, 2012.

\bibitem{gipo_sur}
Alexandre Girouard and Iosif Polterovich.
\newblock Spectral geometry of the {S}teklov problem (survey article).
\newblock {\em J. Spectr. Theory}, 7(2):321--359, 2017.

\bibitem{hho2}
B.~Helffer, M.~Hoffmann-Ostenhof, T.~Hoffmann-Ostenhof, and M.~P. Owen.
\newblock Nodal sets for groundstates of {S}chr\"{o}dinger operators with zero
  magnetic field in non-simply connected domains.
\newblock {\em Comm. Math. Phys.}, 202(3):629--649, 1999.

\bibitem{sir}
Debora Impera, Michele Rimoldi, and Alessandro Savo.
\newblock Index and first {B}etti number of {$f$}-minimal hypersurfaces and
  self-shrinkers.
\newblock {\em Rev. Mat. Iberoam.}, 36(3):817--840, 2020.

\bibitem{ly0}
Peter Li and Shing~Tung Yau.
\newblock A new conformal invariant and its applications to the {W}illmore
  conjecture and the first eigenvalue of compact surfaces.
\newblock {\em Invent. Math.}, 69(2):269--291, 1982.

\bibitem{LoPa}
Rafael L{\'o}pez and {\'A}lvaro P{\'a}mpano.
\newblock Classification of rotational surfaces in {Euclidean} space satisfying
  a linear relation between their principal curvatures.
\newblock {\em Math. Nachr.}, 293(4):735--753, 2020.

\bibitem{nad}
N.~Nadirashvili.
\newblock Berger's isoperimetric problem and minimal immersions of surfaces.
\newblock {\em Geom. Funct. Anal.}, 6(5):877--897, 1996.

\bibitem{shig}
Ichiro Shigekawa.
\newblock Eigenvalue problems for the {S}chr\"{o}dinger operator with the
  magnetic field on a compact {R}iemannian manifold.
\newblock {\em J. Funct. Anal.}, 75(1):92--127, 1987.

\bibitem{steklov}
W.~Stekloff.
\newblock Sur les probl\`emes fondamentaux de la physique math\'ematique (suite
  et fin).
\newblock {\em Ann. Sci. \'Ecole Norm. Sup. (3)}, 19:455--490, 1902.

\bibitem{weinstock}
Robert Weinstock.
\newblock Inequalities for a classical eigenvalue problem.
\newblock {\em J. Rational Mech. Anal.}, 3:745--753, 1954.

\end{thebibliography}
\end{document}